\documentclass[reqno,11pt]{amsart}
\usepackage{amssymb,amsmath,amsfonts,mathrsfs,verbatim}
\usepackage{appendix}

\numberwithin{equation}{section} \theoremstyle{plain}
\theoremstyle{plain}
\newtheorem{thm}[subsection]{Theorem}
\newtheorem{cor}[subsection]{Corollary}
\newtheorem{lem}[subsection]{Lemma}
\newtheorem{prop}[subsection]{Proposition}

\theoremstyle{definition}
\newtheorem{Def}[subsection]{Definition}

\theoremstyle{remark}

\newtheorem{rem}[subsection]{Remark}

\newtheorem{exs}[subsection]{Example}

\newcommand{\loc}{\operatorname{loc}}
\newcommand{\Dom}{\operatorname{Dom}}

\newcommand{\lloc}{L_{\loc}}

\newcommand{\End}{\operatorname{End}}

\newcommand\CC{\mathbb{C}}
\newcommand\RR{\mathbb{R}}

\newcommand\NN{\mathbb{N}}

\author{Ognjen Milatovic}
\address{Department of Mathematics
and Statistics
\\ University of North Florida
\\ Jacksonville, FL
32224 \\ USA}
\email{omilatov@unf.edu}
\author{Hemanth Saratchandran}
\address{Institut f{\"u}r Mathematik \\
Differentialgeometrie \\
Universit{\"a}t Augsburg \\
Universit{\"a}tsstra{\ss}e 14 \\
86159 Augsburg \\
Germany}
\email{hemanth.saratchandran@gmail.com}
\title{Essential self-adjointness of perturbed biharmonic operators via
conformally transformed metrics}

\keywords{biharmonic operator, Bochner Laplacian, conformally transformed metric,
essential self-adjointness, perturbation,
Riemannian manifold}

\subjclass[2010]{35P05, 58J05, 58J50}
\begin{document}
\maketitle
\begin{abstract}
We give sufficient conditions for the essential self-adjointness
of perturbed biharmonic operators acting on sections of a Hermitian
vector bundle over a Riemannian manifold with additional assumptions, such as lower semi-bounded Ricci curvature or bounded sectional curvature. In the case of lower semi-bounded Ricci curvature, we formulate our results in terms of the completeness of the metric that is conformal to the original one, via a conformal factor that depends on a minorant
of the perturbing potential $V$. In the bounded sectional curvature situation, we are able to relax the growth condition on the minorant of $V$ imposed in an earlier article. In this context, our growth condition on the minorant of $V$ is consistent with the literature on the self-adjointness of perturbed biharmonic operators on $\mathbb{R}^n$.
\end{abstract}

\section{Introduction}

The topic of essential self-adjointness of
Schr\"odinger operators on $\RR^n$, along with various applications in mathematical physics, has been studied thoroughly over the past hundred years. The accounts of some of the most important developments are given in the books~\cite{cycon,kato,reed} and in the paper~\cite{Simon-18}. The starting point for the exploration of this theme in the context of Riemannian manifolds is the article \cite{gaffney}, where the essential self-adjointness of the scalar Laplacian and Hodge Laplacian on a
complete Riemannian manifold was established. After about two decades, this result was generalised by the author of \cite{cordes}, who proved the essential self-adjointness of (positive integer) powers of the scalar Laplacian. Around the same time, the author of~\cite{chernoff} used hyperbolic equation techniques to prove the essential self-adjointness of (positive integer) powers of first-order operators, thus incorporating the self-adjointness of powers of the Hodge Laplacian.

The 1990s opened up avenues for the exploration of the self-adjointness problem for Schr{\"o}dinger operators on
Riemannian manifolds (including those acting on sections of a Hermitian vector bundle). This (ongoing) investigation has resulted in quite a few articles, of which we mention~\cite{braverman_1,br-c, braverman_2,grummt,guneysu-17-paper,GP,lm,oleinik,Oleinik94,shubin_1,shubin_2} here and refer the reader to the book~\cite{guneysu-17} for additional references.

Before proceeding further, let us recall (see~\cite{s-biharmonic-book}) that the square of the Laplacian (also known as the biharmonic operator or bi-Laplacian) appears in the study of mechanics of elastic plates, hydrodynamics (slow flows of viscous fluids), and elasticity theory. Naturally, one is lead to the problem of finding sufficient conditions for the (essential) self-adjointness of a perturbation of the biharmonic operator by a potential. An important study in this regard is the paper~\cite{hdn-12} concerning the operators $P+V$, where $P$ is an elliptic differential operator of order $2m$, $m\in\NN$. As a corollary of the principal result of~\cite{hdn-12}, it was established that $\Delta^2+V$, where $\Delta$ is the standard Laplacian on $\mathbb{R}^n$ and $V\in\lloc^{\infty}(\mathbb{R}^n)$, is essentially self-adjoint on $C_{c}^{\infty}(\mathbb{R}^n)$ if $V(x)\geq -q(|x|)$, where $q\geq 1$ is a non-decreasing function such that $q(s)=O(s^{4/3})$. Soon after, the author of~\cite{brusentsev} considered the operator $(-\Delta)^{m}+V$, where $m\in\NN$ and $V\in\lloc^{\infty}(\mathbb{R}^n)$, and showed (here we only describe a special case of his result for $m=2$) that $\Delta^{2}+V$ is essentially self-adjoint if (i) $\Delta^{2}+V$ is semibounded from below on $C_{c}^{\infty}(\mathbb{R}^n)$ and (ii) $V(x)\geq -q(|x|)$, for all $x\in\mathbb{R}^n$, where $q\colon [0,\infty)\to [1,\infty)$ is a $C^2$-function such that $\int_{0}^{\infty}q^{-1/4}(t)\,dt=\infty$, $|(q^{-1/4})'|\leq C$ and $|(q^{-1/4})''|\leq Cq^{1/2}$, for some constant $C$.

One of the difficulties with extending the results from the preceding paragraph to a geodesically complete Riemannian manifold $(M,g)$ stems from the fact that the quadratic form for $\Delta_{g}^2$ is ``driven" by $\Delta_{g}$, the non-negative scalar Laplacian on $M$, which means that localising the problem requires a sequence of Laplacian cut-off functions (described in section~\ref{smooth-cut-off} below). Very recently, the authors of~\cite{bianchi} proved the existence of such a sequence on a geodesically complete Riemannian manifold whose Ricci curvature is lower semibounded by a (possibly unbounded) non-positive function depending on $d_{g}(x_0,\cdot)$, the distance from a fixed reference point $x_0$. Under these geometric assumptions, together with the hypothesis $V(\cdot)\geq -q(d_{g}(x_0,\cdot))$, where $V\in\lloc^{\infty}$ and $q\geq 1$ is a non-decreasing function such that $q(s)=O(s)$, the essential self-adjointness of $\Delta_{B}^2+V$ on $C_{c}^{\infty}(E)$, where $\Delta_{B}$ is the Bochner Laplacian on a Hermitian vector bundle $E$ over $M$, was established in the paper~\cite{milatovic}.

In theorems~\ref{main_thm_1} and \ref{main_thm_2} of the present article (with the latter result pertaining to lower semibounded operators), the sufficient conditions for the (essential) self-adjointness of $L=\Delta_{B}^2+V$, with $\Delta_{B}$ as in the preceding paragraph, are expressed in terms of the completeness of the metrics $\tilde{g}:=Q^{\alpha}g$, where $g$ is the original metric on $M$, $Q\geq 1$ is a smooth function satisfying $V(x)\geq -Q(x)$ for all $x\in M$, and $\alpha=-3/2$ (or $\alpha=-1/2$).  These theorems  complement the main result of~\cite{milatovic} and serve as analogues of the results established in~\cite{braverman_1,lm,Oleinik94,shubin_1,shubin_2} in the setting of Schr\"odinger operators.  The objective of the method used in theorem~\ref{main_thm_1} is to show that the ``maximal operator" $L_{max}:=(L|_{C_{c}^{\infty}})^{*}$ is symmetric (here $T^*$ stands for the adjoint of $T$), while in the situation of theorem~\ref{main_thm_2}, the goal is to show that the closure of $L|_{C_{c}^{\infty}}$ is ``semimaximal" (see definition~\ref{semimax_def} for the meaning of the latter term).  The slight difference in two approaches is due to the presence of lower semiboundedness assumption in theorem~\ref{main_thm_2}, which enables us to use an abstract lemma from~\cite{brusentsev}. The key step in the method, common to both theorems, is to establish the finiteness of the ``energy-type" integrals $\|Q^{-1/4}\nabla u\|$ and $\|Q^{-1/2}\Delta_{B} u\|$ for all $u\in \Dom(L_{max})$, where $\nabla$ is the covariant derivative on $E$ and $\|\cdot\|$ is the $L^2$-norm. This is accomplished through a sequence of estimates relying on the Laplacian cut-off functions (with properties as in section~\ref{cut-off_section}).  After that, we bring in the metric $\tilde{g}=Q^{\alpha}g$ with appropriate $\alpha<0$, a device used in~\cite{shubin_1,shubin_2} for the Schr\"odinger operator situation, and use this metric (which we assume to be complete) in combination with the finiteness of the ``energy-type" integrals to finish the proofs of the theorems. The idea to arrive at the self-adjointness of Schr\"odinger operators (or, more generally, differential operators of order $2m$) on $\mathbb{R}^n$ through the finiteness of the ``energy-type" integrals can be traced back to the papers~\cite{rb-70} and~\cite{brb-74}, respectively, and subsequent adaptations of this idea (with appropriate refinements) to various Schr\"odinger-type operators on manifolds can be found in~\cite{braverman_1,lm,Oleinik94,shubin_1,shubin_2}.

In contrast to the Schr\"odinger operator situation, where the completeness of the transformed metric $\tilde{g}$ is sufficient to guarantee the existence of appropriate ``first-order" cut-off functions, here we need to ensure that, in addition to the completeness of $\tilde{g}$, the Ricci curvature tensor $Ric_{\tilde{g}}$ (obtained by transforming  $Ric_{{g}}$ as in proposition~\ref{Ricci_curvature_conform_trans}) is lower semi-bounded. As seen in the statements of corollaries~\ref{cor1_1} and~\ref{cor1_3}, this imposes some restrictions on $dQ$ and $Hess_{g} (Q)$, the Hessian of $Q$ with respect to $g$. These restrictions, in turn, place additional limitations on the growth of the function $Q$. Nevertheless, in the context of complete manifolds $(M,g)$ with bounded sectional curvature $Sec_{g}$ (see theorems~\ref{main_thm_1-s} and~\ref{main_thm_2-s}, the latter one pertaining to lower semi-bounded operators),  we are able to use the method of ``energy-type" integrals without reliance on the (sectional) curvature transformation formula. Concerning the perturbation $V\in\lloc^{\infty}$, in theorems~\ref{main_thm_1-s} and~\ref{main_thm_2-s} we assume that $V(x) \geq -(f\circ r_{g})(x)$ for all $x\in M$, where $f\colon (0,\infty)\to[1,\infty)$ is a smooth non-decreasing function and $r_{g}$ is the distance corresponding to $g$ from a fixed reference point. Here, the growth of $f$ is controlled by the condition  $\int_{0}^{\infty}f^{-\rho/4}(t)\,dt=\infty$ with $\rho=3$ (respectively, $\rho=1$) and some conditions on the first and second derivative of $f$. Theorems~\ref{main_thm_1-s} and~\ref{main_thm_2-s} lead to examples~\ref{ex-1} and~\ref{ex-2} (the latter example pertains to lower semi-bounded operators), which allow $f(t)=(t+1)^{4/3}$ and $f(t)=(t+1)^{4}$ respectively. Neither of these examples can be handled using theorem 2.1 from~\cite{milatovic}, which imposes the restriction $f(t)=O(t)$. The growth of $f$ in examples~\ref{ex-1} and~\ref{ex-2} is consistent with the conditions on $f$ imposed, respectively, in the articles~\cite{hdn-12} and~\cite{brusentsev} in the context of fourth-order operators on $\mathbb{R}^n$.

Let us outline the contents of the paper. In section
\ref{notation_main_theorems}, we explain the notation we will be using
and define the operators we will be working with. We also state the
main theorems and describe two examples. In section \ref{prelim_lemmas} we establish the key estimates that will be needed in the proofs of the main theorems. In sections \ref{proof_thm_1} and \ref{proof_thm_1-s} we give the proofs of the main results, and we prove the corollaries  accompanying the first and the third theorem in section~\ref{S:cor-1}. Finally, in the appendix, we recall some formulas used in various parts of the paper.

\section{Notation and statement of main theorems}\label{notation_main_theorems}

Throughout this paper, $M$ will denote a smooth connected
Riemannian $n$-manifold without boundary. Given a Riemannian
metric $g$ on $M$, we can form the Laplace-Beltrami operator
on functions, which we will always denote by $\Delta_g$. The symbol
$d\mu$ will indicate the volume form associated to $g$.
The Riemannian metric $g$ defines a norm on all tensor powers
of the tangent bundle and cotangent bundle associated to $M$. We will
denote this norm by $|\cdot|$, and often denote the inner product
by $\langle \cdot , \cdot \rangle$;
the context will make it clear
as to which bundle we are applying it to. When needing to emphasize the metric $g$, we will be using the notations $|\cdot|_{g}$ and $\langle \cdot , \cdot \rangle_{g}$.

Fixing a reference point $x_0\in M$, we define
\begin{equation}\label{riem-dist}
r_{g}(x) := d_g(x_0, x),
\end{equation}
where $d_g$ denotes the Riemannian distance function associated to $g$. Furthermore, for $x_0\in M$ and $\kappa>0$ we define
\begin{equation*}
B_{\kappa}(x_0) := \{x\in M\colon r_{g}(x)<\kappa\}.
\end{equation*}

The symbols $Ric_{g}$ and $Sec_{g}$ will stand for the Ricci curvature and sectional curvature corresponding to the metric $g$.
In this article, the inequality $Ric_{g}\geq-K$ for some constant $K\geq0$ will be understood in the following sense: for all $x$ in $M$ and all $X\in T_{x}M$, we have
$Ric_{g}(X,X)\geq -K\langle X,X\rangle_{g}$.

We will fix a smooth Hermitian vector bundle $(E, h)$ over $M$, with
Hermitian metric $h$. We will also fix a Hermitian connection $\nabla$
on $E$.  We note that the two metrics
$h$ and $g$ can then be used to define metrics on tensor powers
of $E$ with the tangent and cotangent bundles of $M$. We will simply
denote these induced metrics by $\langle\cdot, \cdot\rangle$, and
the associated norms by $|\cdot |$; the context will make it clear as to which bundle we are working on.

The notations $C^{\infty}(M)$, $C^{\infty}_c(M)$ will indicate
smooth functions and smooth functions with compact support on $M$
respectively. Similarly, the symbols $C^{\infty}(E)$ and
$C^{\infty}_c(E)$ will denote smooth sections and smooth sections
with compact support on $E$ respectively.

For  $f\in C^{\infty}(M)$, the symbol $Hess_{g} (f)$ will be understood as
$Hess_{g} (f):=\nabla^{lc,g}df$, where $\nabla^{lc,g}$ is the covariant derivative on $T^*M$ induced from the Levi--Civita connection on $(M,g)$ and $d$ is the standard differential.

The notation  $L^2(E)$ will indicate the Hilbert space of square integrable sections of $E$, with inner product
\begin{equation*}
( u, v) := \int_{M}h(u, v)d\mu.
\end{equation*}
The associated $L^2$-norm will be denoted by
\begin{equation*}
\vert\vert u\vert\vert := \bigg{(}
\int_M\vert u\vert^2d\mu\bigg{)}^{1/2},
\end{equation*}
where $\vert u\vert^2 = h(u, u)$.

The symbol $W^{k,p}_{loc}(E)$, $1\leq p\leq\infty$, will stand for the
local Sobolev spaces of $L^p$-type sections, with $k$ indicating the highest order of
derivatives. For $k=0$, we simply
write $L^{p}_{loc}(E)$. The space of compactly supported elements of $W^{k,p}_{loc}(E)$
will be denoted by $W^{k,p}_{comp}(E)$. In the case $E=M \times \mathbb{C}$, we will use the symbols
$W^{k,p}_{loc}(M)$ and $W^{k,p}_{comp}(M)$.

%We remind the reader that the latter space
%consists of those measurable sections of $End E$ that have
%finite essential supremum almost everywhere,
%over relatively compact open sets.

The formal adjoint of $\nabla$ with respect to $(\cdot,\cdot)$ will be denoted by
$\nabla^{\dagger}$, with the associated Bochner Laplacian being given
by $\Delta_{B} := \nabla^{\dagger}\nabla$. In the case $E=M\times \mathbb{C}$ and $\nabla=d$, the operator $\Delta_{B}$ becomes the usual Laplace-Beltrami operator $\Delta_{g}:=d^{\dagger}d$. We will be studying perturbations of the biharmonic operator, associated
to the Bochner Laplacian on $E$, defined by
\begin{equation*}
L := (\Delta_{B})^2 + V,
\end{equation*}
where $V$ is a linear self-adjoint bundle map in $\lloc^{\infty}(\End E)$.

We now state the principal results of the paper.

\begin{thm}\label{main_thm_1}
Let $(M, g)$ be a Riemannian manifold and  let $(E, h)$ be
a Hermitian vector bundle over $M$, with Hermitian metric $h$.
Fix a Hermitian connection $\nabla$ on $E$. Let $V$ be a
self-adjoint endomorphism of $E$ with $V \in \lloc^{\infty}(\End E)$.
Let $Q\colon M \rightarrow [1, \infty)$ be a function in $C^{\infty}(M)$
satisfying
\begin{itemize}
\item[(i)] $|dQ|_g \leq CQ^{5/4}$,
\item[(ii)] $|\Delta_gQ| \leq CQ^{3/2}$,
\end{itemize}
for all $x\in M$, where $C\geq 0$ is some constant. Suppose $V(x) \geq -Q(x)$ for all $x \in M$. Furthermore, suppose that the metric
$\tilde{g} := Q^{-3/2}g$ is complete. Lastly, assume that $Ric_g \geq -K_1$ and
$Ric_{\tilde{g}} \geq -K_2$ for some constants $K_j \geq 0$, $j=1,2$. Then $L$ is essentially self-adjoint on $C_c^{\infty}(E)$.
\end{thm}

\begin{rem} \label{comment-compl} Keeping in mind the condition $Q\geq 1$, we see that the completeness of $Q^{-3/2}g$ implies the completeness of $g$.
\end{rem}

\begin{rem}\label{rem-cut-off} The assumption of boundedness from below by a constant for $Ric_g$ and $Ric_{\tilde{g}}$ in theorem~\ref{main_thm_1} (and theorem~\ref{main_thm_2} below) may be weakened so that Ricci curvatures $Ric_g$ and $Ric_{\tilde{g}}$ are bounded from below by certain (non-positive) functions (as in corollary of 2.3 of~\cite{bianchi}) depending on the distances $d_{g}(\cdot,x_0)$ and $d_{\tilde{g}}(\cdot,x_0)$ from a fixed reference point $x_0$, respectively.
%For examples in section~\ref{example}, we need simplified versions of those hypotheses, as stated in our theorems.
\end{rem}

\begin{cor}\label{cor1_1} Let $(M,g)$ be a Riemannian manifold and let $E$ be a Hermitian vector bundle over $M$ with Hermitian connection $\nabla$. Assume that $Ric_g\geq 0$. Let $Q : M \rightarrow [1, \infty)$ be a function in $C^{\infty}(M)$ satisfying
\[
|dQ|_g \leq CQ^{1/4},\quad |Hess_{g}(Q)|_g \leq CQ^{-1/2},
\]
for all $x\in M$, where $C\geq 0$ is some constant.   Assume that the metric $Q^{-3/2}g$ is complete and $V\geq -Q$. Then,
$L$ is essentially self-adjoint on $C_c^{\infty}(E)$.
\end{cor}

The next theorem is concerned with manifolds of bounded sectional curvature.

\begin{thm}\label{main_thm_1-s}
Let $(M, g)$ be a Riemannian manifold with $|Sec_{g}|_{g}\leq K_3$, where $K_3\geq 0$ is a constant. Let $(E, h)$ be
a Hermitian vector bundle over $M$, with Hermitian metric $h$.
Fix a Hermitian connection $\nabla$ on $E$. Let $V$ be a
self-adjoint endomorphism of $E$ with $V \in \lloc^{\infty}(\End E)$.
Let $f\colon [0,\infty) \rightarrow [1, \infty)$ be a function
satisfying the following properties:
\begin{itemize}
\item[(i)] $f$ is smooth and non-decreasing,
\item[(ii)] $\int_{0}^{\infty}f^{-3/4}(t)\,dt=\infty$,
\item[(iii)] $|f'(t)| \leq Cf^{5/4}(t)$, for all $t\geq 0$,
\item[(iv)] $|f''(t)| \leq Cf^{3/2}(t)$, for all $t\geq 0$,
\end{itemize}
where $C\geq 0$ is some constant. Suppose $V(x) \geq -(f\circ r_{g})(x)$ for all $x\in M$, where $r_{g}$ is as in~(\ref{riem-dist}). Then $L$ is is essentially self-adjoint on $C_c^{\infty}(E)$.
\end{thm}

The following example illustrates theorem~\ref{main_thm_1-s}:

\begin{exs}\label{ex-1} Let $(M, g)$ be a complete Riemannian manifold satisfying the property $|Sec_g|_{g}\leq K_{3}$, where $K_{3}\geq 0$ is some constant. Let $E$ be a Hermitian vector bundle over $M$ with a Hermitian connection $\nabla$. Let $f(t)=(t+1)^{4/3}$, where $t\geq 0$. Suppose that $V$ is a self-adjoint endomorphism such that $V \in \lloc^{\infty}(\End E)$ and $V(x)\geq -f(r_{g}(x))$, for all $x\in M$, where $r_{g}$ is as in~(\ref{riem-dist}).  It is easy to check that the function $f$ satisfies the hypotheses of theorem~\ref{main_thm_1-s}. Therefore, $L$ is essentially self-adjoint on $C_c^{\infty}(E)$. We point out that for this example we cannot use theorem 2.1 from~\cite{milatovic}, which imposes the requirement $f(t)=O(t)$. $\hfill\square$
\end{exs}

The last two theorems pertain to lower semi-bounded operators.

\begin{thm}\label{main_thm_2}
Assume that the hypotheses of theorem~\ref{main_thm_1} are satisfied with the following change: instead of the completeness of the metric $Q^{-3/2}g$, assume the completeness of $\tilde{g}:=Q^{-1/2}g$. (Now $\tilde{g}$ in the notation $Ric_{\tilde{g}}$ refers to $\tilde{g}=Q^{-1/2}g$). In addition to the hypothesis $V\geq -Q$, with $Q$ as in theorem~\ref{main_thm_1}, assume that there exists a constant $K_4\geq 0$ such that
\begin{equation}\label{E:semi-cond-1}
(Lu,u)\geq -K_4\|u\|^2
\end{equation}
for all $C_c^{\infty}(E)$. Then $L$ essentially self-adjoint on $C_c^{\infty}(E)$.
\end{thm}

\begin{rem} Keeping in mind the condition $Q\geq 1$, we see that the completeness of the metric $Q^{-1/2}g$ does not imply the completeness of $Q^{-3/2}g$. Therefore, theorem~\ref{main_thm_2} is not contained in theorem~\ref{main_thm_1}.
\end{rem}

\begin{cor}\label{cor1_3} Let $(M,g)$ be a Riemannian manifold and let $E$ be a Hermitian vector bundle over $M$ with Hermitian connection $\nabla$. Assume that $Ric_g\geq 0$. Let $Q : M \rightarrow [1, \infty)$ be a function in $C^{\infty}(M)$ satisfying
\[
|dQ|_g \leq CQ^{3/4},\quad |Hess_{g}(Q)|_g \leq CQ^{1/2},
\]
for all $x\in M$, where $C\geq 0$ is some constant.  Assume that the metric $Q^{-1/2}g$ is complete and $V\geq -Q$. Furthermore, assume
that $L$ satisfies~(\ref{E:semi-cond-1}). Then
$L$ is essentially self-adjoint on $C_c^{\infty}(E)$.
\end{cor}

\begin{thm}\label{main_thm_2-s}
Assume that the hypotheses of theorem~\ref{main_thm_1-s} are satisfied with the following change: instead of the condition (ii), assume that
\begin{equation*}
\int_{0}^{\infty}f^{-1/4}(t)\,dt=\infty.
\end{equation*}
In addition to the hypothesis $V(x) \geq -(f\circ r_{g})(x)$, where $r_{g}$ is as in~(\ref{riem-dist}), assume that~(\ref{E:semi-cond-1}) is satisfied. Then $L$ essentially self-adjoint on $C_c^{\infty}(E)$.
\end{thm}

The following example illustrates theorem~\ref{main_thm_2-s}:

\begin{exs}\label{ex-2} Here we use the same description as in example~\ref{ex-1} with the following change: let $f(t)=(t+1)^{4}$, where $t\geq 0$.
In addition to the hypothesis $V(x) \geq -f(r_{g}(x))$, assume that~(\ref{E:semi-cond-1}) is satisfied. Again, one easily checks that $f$ satisfies the hypotheses of theorem~\ref{main_thm_2-s}. Therefore, $L$ is essentially self-adjoint on $C_c^{\infty}(E)$. As in the case of example~\ref{ex-1}, we cannot use theorem 2.1 from~\cite{milatovic} here.
 $\hfill\square$
\end{exs}

\section{Preliminary lemmas}\label{prelim_lemmas}

\subsection{Cut-off functions}\label{cut-off_section} In this section, $M$ is a Riemannian manifold with metric $g$. In order to obtain suitable localised derivative estimates, we will assume that $(M,g)$ is equipped with a sequence of cut-off functions $\{\chi_k\}$, $k\in\NN$, with the following properties:

\begin{itemize}
\item[(e1)] For all $x \in M$ and all $k \in \NN$, we have
$0 \leq \chi_k(x) \leq 1$.

\item[(e2)] $\chi_k\in W^{4,\infty}_{comp}(M)$, where, in the Sobolev space notation, $4$ indicates the highest derivative and $\infty$ indicates the $L^{\infty}$-space.

\item[(e3)] For all $x\in M$, we have $\displaystyle\lim_{k\to\infty}\chi_k(x)=1$.

\item[(e4)]  If $Q\in C^{\infty}(M)$ is a function such that $Q\geq 1$ and if we set $\psi_k:= \chi_k^4Q^{-1/2}$,
then $\psi_k$ satisfies the following estimates:
\begin{equation*}
|d\psi_k|_{g} \leq C\chi_k^3Q^{-1/4}\text{ and } |\Delta_{g}\psi_k| \leq C\chi_k^2,
\end{equation*}
where $C>0$ is a constant independent of $k$.
\end{itemize}

\begin{rem}\label{gb-e}
As we will see later, the existence of a sequence $\{\chi_k\}$ with properties (e1)--(e4) is guaranteed under additional geometric assumptions on $(M,g)$ (such as completeness of $g$ and boundedness from below of $Ric_g$) and additional requirements on $Q$.
\end{rem}

\begin{rem}\label{chi-k-q-rem}
With $Q$ and $\{\psi_k\}$ as in (e4), we see that $0\leq \psi_k \leq 1$ and $\psi_k \in W^{4,\infty}_{comp}(M)$. Furthermore, if $u\in W^{4,2}_{loc}(E)$, then $\psi_k u\in W^{4,2}_{comp}(E)$.
\end{rem}

To simplify the notations in the remainder of the section, the symbol $|\cdot|$ refers to the norm with respect to $g$ or the metrics induced on $T^*M$ and $T^*M\otimes E$ from $g$ and the metric $h$ of the bundle $E$.

\subsection{Key estimates}\label{key_ests}
Here we establish the key estimates used in the proofs of the main theorems of this paper.
For all lemmas of this section, we assume that $M$ is a Riemannian manifold with metric $g$. Additionally, we assume that $Q\colon M\to [1,\infty)$ is a function belonging to $C^{\infty}(M)$.  Furthermore, we assume that $(E, h)$ is
a Hermitian vector bundle over $M$, with Hermitian metric $h$ and Hermitian connection $\nabla$. We also assume that $V$ is a self-adjoint endomorphism of $E$ with $V \in \lloc^{\infty}(\End E)$ and, starting from lemma~\ref{J_2^k-estimate}, we assume that $V(x)\geq -Q(x)$ for all $x\in M$.

We start with some operator theoretic preliminaries. We define the minimal operator associated to $L=(\Delta_{B})^2 + V$, where $\Delta_{B}=\nabla^{\dagger}\nabla$,  by
$L_{min} := \Delta_{B}^2 + V$ with $Dom(L_{min}) = C^{\infty}_c(E)$.
We define the maximal operator associated to $L$ by
$L_{max} := (L_{min})^*$. It is well known that the operator $L_{max}$ can
be described as follows: $L_{max}u = Lu$ with
$Dom(L_{max}) = \{u \in L^2(E) : Lu \in L^2(E)\}$, where $Lu$ is understood in the sense of distributional sections.

We remark that
$Dom(L_{max}) \subseteq W^{4,2}_{loc}(E)$, which follows from
elliptic regularity.

For $u \in Dom(L_{max})$, we define the following functionals:
\begin{equation*}
J_0[u]:= ||u||,\qquad
J_1[u]:= ||Q^{-1/4}\nabla u||, \qquad
J_2[u]:= ||Q^{-1/2}\Delta_{B} u||.
\end{equation*}

The key point of this section is to prove the finiteness of
$J_1[u]$ and $J_2[u]$, when $u \in Dom(L_{max})$. In order to do this, we will make use of localised versions of $J_1$ and $J_2$:
\begin{equation*}
J_1^{(k)}[u] := ||\chi_{k}^3Q^{-1/4}\nabla u||,\qquad
J_2^{(k)}[u] := ||\chi_{k}^4Q^{-1/2}\Delta_{B} u||,
\end{equation*}
where
$u \in Dom(L_{max})$, and the functions
$\{\chi_k\}$ are as in section~\ref{cut-off_section}.

Let $\psi_k$ be as in~(e4). We start by observing that for all $u \in Dom(L_{max})$, we have the following formula:
\begin{align}
(L(\psi_ku), \psi_ku) &= (\Delta_{B}^2(\psi_ku), \psi_ku) +(V (\psi_ku), \psi_ku)
= (\Delta_{B}(\psi_ku), \Delta_{B}(\psi_ku)) + (Vu,\psi_k^2u)\nonumber\\
&=(\Delta_{B}(\psi_ku), \Delta_{B}(\psi_ku))+(Lu,\psi_k^2u)-(\Delta_{B}^2u,\psi_k^2u)\nonumber,
\end{align}
where in the second equality we used integration by parts, which is allowed because (see remark~\ref{chi-k-q-rem}) we have $\psi_ku\in W^{4,2}_{comp}(E)$. After taking real parts on both sides, we obtain
\begin{equation}\label{M-function_form_1}
(L(\psi_ku), \psi_ku) =Re(\psi_k^2Lu, u) + P_{\psi_k}[u],
\end{equation}
where $P_{\psi_k}[u] := (\Delta_{B}(\psi_ku), \Delta_{B}(\psi_ku)) -
Re(\Delta_{B}^2u, \psi_k^2u)$.

We then have the following lemma:
\begin{lem}\label{M-function_form_2} Let $\psi_k$ be as in~(e4). Then,
for all $u \in Dom(L_{max})$ we have
\begin{equation*}
P_{\psi_k}[u] = 4||\nabla_{(d\psi_k)^\#}u||^2 + ||u\Delta_{g}\psi_k||^2
-4Re(\nabla_{(d\psi_k)^\#}u, u\Delta_{g}\psi_k) +
2Re(\Delta_{B} u, u|d\psi_k|^2),
\end{equation*}
where $(d\psi_k)^\#$ is the vector field corresponding to the form $d\psi_k$ via the metric $g$.
\end{lem}
\begin{proof}
Using the product rule for the Laplacian, proposition
\ref{adjoint_product_formula}(ii), we find
\begin{align}\label{ref-prod-rule-1}
(\Delta_{B}(\psi_ku), \Delta_{B}(\psi_ku)) &= (\psi_k\Delta_{B} u -2\nabla_{(d\psi_k)^\#}u + u\Delta_{g}\psi_k,
\psi_k\Delta_{B} u - 2\nabla_{(d\psi_k)^\#}u + u\Delta_{g}\psi_k)\nonumber \\
&= (\psi_k\Delta_{B} u, \psi_k\Delta_{B} u) +
4(\nabla_{(d\psi_k)^\#}u, \nabla_{(d\psi_k)^\#}u) +
(u\Delta_{g}\psi_k, u\Delta_{g}\psi_k)\nonumber \\
&\hspace{0.5cm}-
4Re(\psi_k\Delta_{B} u, \nabla_{(d\psi_k)^\#}u) +
2Re(\psi_k\Delta_{B} u, u\Delta_{g}\psi_k) \nonumber\\
&\hspace{0.5cm}-4Re(\nabla_{(d\psi_k)^\#}u, u\Delta_{g}\psi_k).
\end{align}
We also have
\begin{align*}
Re(\Delta_{B}^2u, \psi_k^2u) &= Re(\Delta_{B} u, \Delta_{B}(\psi_k^2u))
= Re(\Delta_{B} u, \psi_k^2\Delta_{B} u - 2\nabla_{(d\psi^2_k)^\#}u +u\Delta_{g}\psi^2_k)\\
&= (\psi_k\Delta_{B} u, \psi_k \Delta_{B} u) -
4Re(\Delta_{B} u, \psi_k\nabla_{(d\psi_k)^\#}u) +
2Re(\Delta_{B} u, u\psi_k\Delta_{g}\psi_k) \\
&\hspace{0.5cm}- 2Re(\Delta_{B} u, u|d\psi_k|^2),
\end{align*}
where in the last equality we used the formulas $d\psi^2_k=2\psi_kd\psi_k$ and $\Delta_{g}\psi^2_k=2\psi_k\Delta_{g}\psi_k-2|d\psi_k|^2$.
Using these two computations, we then see that
\begin{align*}
P_{\psi_k}[u] &= (\Delta_{B}(\psi_ku), \Delta_{B}(\psi_ku)) -
Re(\Delta_{B}^2u, \psi_k^2u) \\
&= 4||\nabla_{(d\psi_k)^\#}u||^2 + ||u\Delta_g\psi_k||^2
- 4Re(\nabla_{(d\psi_k)^\#}u, u\Delta_g\psi_k) + 2Re(\Delta_{B} u, u|d\psi_k|^2).
\end{align*}
\end{proof}

\begin{lem}\label{M-function_estimate} Let $\psi_k$ be as in~(e4). Then,
for all $u \in Dom(L_{max})$ we have
\begin{equation*}
|P_{\psi_k}[u]| \leq C\bigg{(}(J_0[u])^2 + J_0[u]J_1^{(k)}[u] + (J_1^{(k)}[u])^2 + J_0[u]J_2^{(k)}[u]
\bigg{)},
\end{equation*}
where $C>0$ is a constant independent of $k$ and $u$.
\end{lem}
\begin{proof}
Using the triangle inequality, from lemma~\ref{M-function_form_2} we get
\begin{equation*}
|P_{\psi_k}[u]| \leq 4||\nabla_{(d\psi_k)^\#}u||^2 + ||u\Delta_g\psi_k||^2
+ 4|(\nabla_{(d\psi_k)^\#}u, u\Delta_g\psi_k)| +
2|(\Delta_{B} u, u|d\psi_k|^2)|.
\end{equation*}
We then estimate each term on the right hand side of the above
inequality.

We start with the term  $||\nabla_{(d\psi_k)^\#}u||^2$ and obtain
\begin{align*}
||\nabla_{(d\psi_k)^\#}u||^2 &= \int |\nabla_{(d\psi_k)^\#}u|^2d\mu
\leq \int |d\psi_k|^2|\nabla u|^2d\mu
\leq \int C\chi_k^6Q^{-1/2}|\nabla u|^2d\mu
= C(J^{k}_1[u])^2,
\end{align*}
where to get the second inequality we
used~(e4), and to get the first inequality we have used the fact that
\begin{equation*}
|\nabla_{(d\psi_k)^\#}u| = |tr(d\psi_k \otimes \nabla u)|
\leq |d\psi_k||\nabla u|.
\end{equation*}
The term $||u\Delta_{g}\psi_k||^2$ is estimated by
\begin{align*}
||u\Delta_g\psi_k||^2 = \int |u\Delta_g\psi_k|^2d\mu
\leq C||u||^2
=
C(J_0[u])^2,
\end{align*}
where we used~(e1) and (e4) to get the
inequality.

Using~(e4) together with Cauchy--Schwarz inequality,
the term $4|(\nabla_{(d\psi_k)^\#}u, u\Delta_g\psi_k)|$ is
estimated by
\begin{align*}
4|(\nabla_{(d\psi_k)^\#}u, u\Delta_g\psi_k)| &\leq
4\int |d\psi_k||\nabla u||u||\Delta_g\psi_k|d\mu
\leq
C\int\chi_k^3Q^{-1/4}|\nabla u||u|d\mu \\
&\leq
C\bigg{(} \int\chi_k^6Q^{-1/2}|\nabla u|^2d\mu\bigg{)}^{1/2}
\bigg{(} \int|u|^2d\mu\bigg{)}^{1/2}
= C(J_1^{(k)}[u])(J_0[u]),
\end{align*}
where the  second inequality is obtained by using the property $\chi_k^5\leq
\chi_k^3$.

Finally, the term $2|(\Delta_{B} u, u|d\psi_k|^2)|$ is estimated analogously to the preceding one:
\begin{align*}
2|(\Delta_{B} u, u|d\psi_k|^2)| &\leq C\int|u|Q^{-1/2}\chi_k^6|\Delta_{B} u|d\mu
 \leq C\bigg{(}\int |u|^2 d\mu\bigg{)}^{1/2}
\bigg{(}\int Q^{-1}\chi_k^8|\Delta_{B} u|^2 d\mu\bigg{)}^{1/2} \\
&= C(J_0[u])(J_2^{(k)}[u]),
\end{align*}
where in the second inequality we used the property $\chi_k^6\leq \chi_k^4$.
\end{proof}

\begin{lem}\label{L-estimate} Let $\psi_k$ be as in~(e4). Then,
for all $u \in Dom(L_{max})$ we have
\begin{equation*}
(L(\psi_ku), \psi_ku) \leq C\bigg{(}||Lu||||u|| + (J_0[u])^2 +
(J_1^{(k)}[u])^2 + J_0[u]J_1^{(k)}[u] + J_0[u]J_2^{(k)}[u]\bigg{)},
\end{equation*}
where $C>0$ is a constant independent of $k$ and $u$.
\end{lem}
\begin{proof}
By \eqref{M-function_form_1} we have
\begin{equation*}
(L(\psi_ku), \psi_ku) = Re(\psi_k^2Lu, u) + P_{\psi_k}[u].
\end{equation*}
We then estimate the first term, on the right hand side of the
above equation, using Cauchy--Schwarz inequality and the fact that
$\psi_k^2 \leq 1$. The term $P_{\psi_k}[u]$ is estimated using
lemma \ref{M-function_estimate}.
\end{proof}

\begin{lem}\label{N-function_lower_bound} Let $\psi_k$ be as in~(e4). Then,
for all $u \in Dom(L_{max})$ we have
\begin{equation*}
(\Delta_{B}(\psi_ku), \Delta_{B}(\psi_ku)) \geq (\psi_k^2\Delta_{B} u, \Delta_{B} u)
- CJ_2^{(k)}[u]J_1^{(k)}[u] - CJ_1^{(k)}[u]J_0[u] - CJ_2^{(k)}[u]J_0[u],
\end{equation*}
where $C>0$ is a constant independent of $k$ and $u$.
\end{lem}
\begin{proof}
Looking at~(\ref{ref-prod-rule-1}) we obtain the simple bound
\begin{align}
(\Delta_{B}(\psi_ku), \Delta_{B}(\psi_ku)) &\geq
(\psi_k\Delta_{B} u, \psi_k\Delta_{B} u)
-
4Re(\psi_k\Delta_{B} u, \nabla_{(d\psi_k)^\#}u) +
2Re(\psi_k\Delta_{B} u, u\Delta_g\psi_k) \nonumber \\
&\hspace{0.5cm} -
4Re(\nabla_{(d\psi_k)^\#}u, u\Delta_g\psi_k). \label{N-function_bound}
\end{align}
We can further estimate the last three terms, on the right hand side
of the above equation, using~(e4). We begin with
\begin{align*}
\vert (\psi_k\Delta_{B} u, \nabla_{(d\psi_k)^\#}u) \vert &\leq
C\int |\psi_k||\Delta_{B} u||d\psi_k||\nabla u|d\mu
\leq C\int Q^{-1/2}\chi_k^4|\Delta_{B} u|Q^{-1/4}\chi_k^3|\nabla u|d\mu \\
&\leq C\bigg{(}\int\chi_k^8Q^{-1}|\Delta_{B} u|^2 d\mu\bigg{)}^{1/2}
\bigg{(}\int \chi_k^6Q^{-1/2}|\nabla u|^2d\mu\bigg{)}^{1/2}
= CJ_2^{(k)}[u]J_1^{(k)}[u],
\end{align*}
where the third inequality follows by applying
Cauchy--Schwarz inequality.

We also have
\begin{align*}
|(\psi_k\Delta_{B} u, u\Delta_{g}\psi_k)| &\leq C\int |\psi_k\Delta_{B} u||u|d\mu
\leq C\bigg{(}\int\psi_k^2|\Delta_{B} u|^2d\mu \bigg{)}^{1/2}
\bigg{(}\int|u|^2 d\mu\bigg{)}^{1/2} \\
&= CJ_2^{(k)}[u]J_0[u],
\end{align*}
where in the first inequality we used $0\leq\chi_k\leq 1$.

Finally, we have
\begin{align*}
|(\nabla_{(d\psi_k)^\#}u, u\Delta_g\psi_k)| &\leq
C\int \chi_k^3Q^{-1/4}|\nabla u||u|d\mu
\leq CJ_1^{(k)}[u]J_0[u],
\end{align*}
where in the first inequality we used the property $\chi_k^5\leq \chi_k^3$. Substituting these estimates into \eqref{N-function_bound} gives the
result.
\end{proof}

In the next lemma, we will use the assumption $V\geq -Q$, where $Q$ is as in (e4).

\begin{lem}\label{J_2^k-estimate} For
for all $u \in Dom(L_{max})$ we have
\begin{equation*}
(J_2^{(k)})^2 \leq C(||Lu||||u|| + (J_0[u])^2 + (J_1^{(k)}[u])^2 +
2J_1^{(k)}[u]J_0[u] + 2J_0[u]J_2^{(k)}[u] +
J_1^{(k)}[u]J_2^{(k)}[u]),
\end{equation*}
where $C>0$ is a constant independent of $k$ and $u$.
\end{lem}
\begin{proof}
Since $V+Q\geq 0$ we have
\[
(L(\psi_ku), \psi_ku) + (Q(\psi_ku), \psi_ku) \geq
(\Delta_B(\psi_ku), \Delta_B(\psi_ku)).
\]
Using the last inequality, lemma \ref{L-estimate}, and the
fact that $Q\psi_k^2=\chi^8_k \leq 1$, we have
\begin{align*}
(\Delta_{B}(\psi_ku), \Delta_{B}(\psi_ku)) &\leq (L(\psi_ku), \psi_ku) + \|u\|^2\nonumber\\
&\leq C\bigg{(}||Lu||||u|| + (J_0[u])^2 +
(J_1^{(k)}[u])^2 + J_0[u]J_1^{(k)}[u] + J_0[u]J_2^{(k)}[u]\bigg{)}.
\end{align*}
We also know, by lemma \ref{N-function_lower_bound}, that
\begin{equation*}
(\Delta_{B}(\psi_ku), \Delta_{B}(\psi_ku)) \geq
(\psi_k^2\Delta_{B} u, \Delta_{B} u)
- CJ_2^{(k)}[u]J_1^{(k)}[u] - CJ_1^{(k)}[u]J_0[u] - CJ_2^{(k)}[u]J_0[u],
\end{equation*}
from which it follows that
\begin{equation*}
(\psi_k^2\Delta_{B} u, \Delta_{B} u) \leq
C\bigg{(}||Lu||||u|| + (J_0[u])^2 +
(J_1^{(k)}[u])^2 + 2J_0[u]J_1^{(k)}[u] + 2J_0[u]J_2^{(k)}[u] +
J_1^{(k)}[u]J_2^{(k)}[u]\bigg{)}.
\end{equation*}
We then note that
\begin{equation*}
(\psi_k^2\Delta_{B} u, \Delta_{B} u) = \int\psi_k^2|\Delta_{B} u|^2d\mu =
\int Q^{-1}\chi_k^8|\Delta_{B} u|^2d\mu = (J_2^{(k)}[u])^2.
\end{equation*}
The result then follows.
\end{proof}

\begin{lem} For all $u \in Dom(L_{max})$ we have
\begin{equation*}
(J_1^{(k)}[u])^2 \leq J_2^{(k)}[u]J_0[u] + CJ_1^{(k)}[u]J_0[u],
\end{equation*}
where $C>0$ is a constant independent of $k$ and $u$.
\end{lem}
\begin{proof}
Keeping in mind that $\psi_k=\chi_k^4Q^{-1/2}$, we can write
\begin{align*}
(J_1^{(k)}[u])^2 &= \int Q^{-1/2}\chi_k^6|\nabla u|^2d\mu
= \int\psi_k\chi_k^2|\nabla u|^2d\mu
\leq \int\psi_k|\nabla u|^2d\mu
= (\psi_k\nabla u, \nabla u) \\
&= (\nabla^{\dagger}(\psi_k\nabla u), u)
= (\psi_k\Delta_{B} u, u) - (\nabla_{(d\psi_k)^\#}u, u),
\end{align*}
where to get the fourth equality we have used integration by parts, and
to get the last equality we have used proposition \ref{adjoint_product_formula}(i).

We then estimate the terms on the right hand side of the above last
equality using Cauchy--Schwarz inequality:
\begin{align*}
|(\psi_k\Delta_{B} u, u)| &\leq \int \psi_k|\Delta_{B} u||u|d\mu
\leq \bigg{(}\int\psi_k^2|\Delta_{B} u|^2d\mu\bigg{)}^{1/2}
\bigg{(}\int |u|^2d\mu\bigg{)}^{1/2}
= J_2^{(k)}[u]J_0[u].
\end{align*}

Using~(e4) and Cauchy--Schwarz inequality we have
\begin{align*}
|(\nabla_{(d\psi_k)^\#}u, u)| &\leq C\int Q^{-1/4}\chi_k^3|\nabla u||u|d\mu
\leq C\bigg{(}\int Q^{-1/2}\chi_k^6|\nabla u|^2d\mu \bigg{)}^{1/2}
\bigg{(}\int |u|^2d\mu \bigg{)}^{1/2} \\
&\leq CJ_1^{(k)}[u]J_0[u].
\end{align*}
These two estimates then give the statement of the lemma.
\end{proof}

\begin{lem}\label{J_1^k-epsilon}
For all $\varepsilon> 0$, there exists a constant $G_{\varepsilon}>0$ (depending on $\varepsilon$ but independent of $k$) such that
\begin{equation*}
(J_1^{(k)}[u])^2 \leq \varepsilon (J_2^{(k)}[u])^2 +
G_{\varepsilon}(J_0[u])^2,
\end{equation*}
for all $u \in Dom(L_{max})$.
\end{lem}
\begin{proof}
By the previous lemma and Young's inequality, we have
\begin{align*}
(J_1^{(k)}[u])^2 &\leq J_2^{(k)}[u]J_0[u] + CJ_1^{(k)}[u]J_0[u]
\leq
{\alpha}(J_2^{(k)}[u])^2 + \frac{1}{4\alpha}(J_0[u])^2 +
{\alpha}(J_1^{(k)}[u])^2 + \frac{C^2}{4\alpha}(J_0[u])^2
\end{align*}
where $\alpha > 0$.

We then have
\begin{equation*}
(1-\alpha)(J_1^{(k)}[u])^2 \leq\alpha(J_2^{(k)}[u])^2 +
\bigg{(}\frac{C^2}{4\alpha} + \frac{1}{4\alpha} \bigg{)}(J_0[u])^2
\end{equation*}
which implies
\begin{equation*}
(J_1^{(k)}[u])^2 \leq \frac{\alpha}{1-\alpha}(J_2^{(k)}[u])^2 +
\bigg{(}\frac{C^2 + 1}{4(1-\alpha)\alpha} \bigg{)}(J_0[u])^2.
\end{equation*}
Now, given $\varepsilon > 0$ let $\alpha := \frac{\varepsilon}{1+\varepsilon}$.
Substituting this value of $\alpha$ into the above equation, gives
\begin{equation}
(J_1^{(k)}[u])^2 \leq \varepsilon(J_2^{(k)}[u])^2 +
\bigg{(}\frac{(C^2 + 1)(1+\varepsilon)^2}{4\varepsilon} \bigg{)}(J_0[u])^2.
\end{equation}
Defining $G_{\varepsilon} := \frac{(C^2 + 1)(1+\varepsilon)^2}{4\varepsilon}$ gives
the result.
\end{proof}

We can now prove the finiteness of $J_2[u]$.

\begin{prop}\label{J_2_finite} Let $M$ be a Riemannian manifold with metric $g$. Let $Q\colon M \to[1,\infty)$ be a function belonging to $C^{\infty}(M)$. Assume that $M$ is equipped with a sequence of functions $\{\chi_k\}$ satisfying the properties (e1)--(e4). Let $(E, h)$ denote
a Hermitian vector bundle over $M$, with Hermitian metric $h$ and Hermitian connection $\nabla$. Let $V$ be a self-adjoint endomorphism of $E$ with $V \in \lloc^{\infty}(\End E)$. Furthermore, assume that $V(x)\geq -Q(x)$ for all $x\in M$. Then, for all $u \in Dom(L_{max})$ we have $J_2[u] < \infty$.
\end{prop}
\begin{proof} Let $\varepsilon>0$ be arbitrary. Using Young's inequality we have
\begin{equation*}
2J_0[u]J_1^{(k)}[u]\leq (J_1^{(k)}[u])^2+(J_0[u])^2
\end{equation*}
\begin{equation*}
2J_0[u]J_2^{(k)}[u]\leq \varepsilon(J_2^{(k)}[u])^2+ \frac{1}{\varepsilon}(J_0[u])^2,
\end{equation*}

\begin{equation*}
J_1^{(k)}[u]J_2^{(k)}[u]\leq \sqrt{\varepsilon}(J_2^{(k)}[u])^2+ \frac{1}{4\sqrt{\varepsilon}}(J_1^{(k)}[u])^2\leq
\frac{5\sqrt{\varepsilon}}{4}(J_2^{(k)}[u])^2+ \frac{G_{\varepsilon}}{4\sqrt{\varepsilon}}(J_0[u])^2,
\end{equation*}
where in the last inequality we used lemma~\ref{J_1^k-epsilon}.

From the last three estimates and lemmas \ref{J_2^k-estimate} and \ref{J_1^k-epsilon}, we obtain
\begin{align*}
(J_2^{(k)}[u])^2 &\leq
C(||Lu||||u|| + (J_0[u])^2 + (J_1^{(k)}[u])^2 +
2J_1^{(k)}[u]J_0[u] + 2J_0[u]J_2^{(k)}[u] +
J_1^{(k)}[u]J_2^{(k)}[u]) \\
&\leq
C[||Lu||||u|| + (3\varepsilon+\frac{5\sqrt{\varepsilon}}{4})(J_2^{(k)}[u])^2 + (2+\frac{1}{\varepsilon}+2G_{\varepsilon} + \frac{G_{\varepsilon}}{4\sqrt{\varepsilon}})(J_0[u])^2],
\end{align*}
where $G_{\varepsilon}$ is as in lemma~\ref{J_1^k-epsilon}.

Making $\varepsilon$ sufficiently small so that, after rearranging, the coefficient of $(J_2^{(k)}[u])^2$ is positive and using the fact that $J_0[u] = ||u||$, we get
\begin{equation*}
(J_2^{(k)}[u])^2 \leq \tilde{C}(||Lu||||u|| + ||u||^2),
\end{equation*}
where $\tilde{C}>0$ is a constant independent of $k$ and $u$.

Taking $k \rightarrow \infty$ in the above inequality, appealing
to (e3), and Fatou's lemma gives the result.
\end{proof}

We can now prove the finiteness of $J_1[u]$.

\begin{prop}\label{J_1_finite} Assume that the hypotheses of proposition~\ref{J_2_finite} are satisfied. Then, for all $u \in Dom(L_{max})$ we have $J_1[u] < \infty$.
\end{prop}
\begin{proof}
For all $\varepsilon>0$ we have
\begin{align*}
(J_1^{(k)}[u])^2 &\leq \varepsilon (J_2^{(k)}[u])^2 +
G_{\varepsilon}(J_0[u])^2
\leq \varepsilon (J_2[u])^2 + G_{\varepsilon}(J_0[u])^2,
\end{align*}
where the first inequality follows from lemma \ref{J_1^k-epsilon}, and for the second inequality we used the property $\chi_k^8\leq 1$.

The right hand side of the above second inequality is finite by
proposition \ref{J_2_finite}. To finish the proof, we let $k \rightarrow \infty$ and use (e3) together with Fatou's lemma.
\end{proof}

Before stating the next two propositions, we assume that in addition to $\{\chi_k\}$ as in section~\ref{cut-off_section}, $M$ is equipped with a sequence of cut-off functions $\{\xi_k\}$ satisfying (e1)--(e3) and the following property:
\begin{itemize}
\item[(e5)] there exists a sequence $p_k\in\mathbb{R}_{+}$ with $\displaystyle\lim_{k\to\infty}p_k=0$ such that
\begin{equation*}
|d\xi_k|_{g} \leq p_kQ^{-\rho/4}\text{ and } |\Delta_{g}\xi_k| \leq p_kQ^{-(\rho-1)/4},
\end{equation*}
where $\rho\geq 1$ and $Q$ is as in (e4).
\end{itemize}
\begin{rem}\label{prop-rho-3-1} In the next two propositions, we will use the condition (e5) with $\rho=3$ and $\rho=1$. Keeping in mind that $Q\geq 1$, we make the following observation: if (e5) is satisfied for $\rho=3$, then it is satisfied for $\rho=1$.
\end{rem}

\begin{prop}\label{prop-key-ests} Let $M$ be a Riemannian manifold with metric $g$. Let $Q\colon M \to[1,\infty)$ be a function belonging to $C^{\infty}(M)$. Assume that $M$ is equipped with a sequence of functions $\{\chi_k\}$ satisfying the properties (e1)--(e4) and a sequence of functions $\{\xi_k\}$ satisfying (e1)--(e3) and (e5) with $\rho=3$. Let $(E, h)$ denote
a Hermitian vector bundle over $M$, with Hermitian metric $h$ and Hermitian connection $\nabla$. Let $V$ be a
self-adjoint endomorphism of $E$ with $V \in \lloc^{\infty}(\End E)$. Furthermore, assume that $V(x)\geq -Q(x)$ for all $x\in M$. Then, $L$ is essentially self-adjoint on $C_{c}^{\infty}(E)$.
\end{prop}
\begin{proof} Keeping in mind the definition $L_{max}:=(L_{min})^*$, in order to establish the essential self-adjointness of $L_{min}$, it is enough to show (in view of an abstract fact) that $L_{\max}$ is a symmetric operator. In other words, we need to show that
\begin{equation}\label{symm-max}
(u, L_{max}v) = (L_{max}u, v),
\end{equation}
for all $u,v\in  Dom(L_{max})$.

%To obtain the estimates on $|d\xi_k|_{{g}}$ and $|\Delta_{{g}}\xi_k|$ we will use proposition
%\ref{d_conf_trans_est} and corollary \ref{Lapl_conf_trans_bound}.

Let $u, v \in Dom(L_{max})$ be arbitrary. Since $M$ is equipped with a sequence of functions $\{\chi_k\}$ satisfying (e1)--(e4), we can use  propositions~\ref{J_2_finite} and~\ref{J_1_finite} to conclude that $J_j[u]<\infty$ and $J_{j}[v]<\infty$ for $j=1,2$.

Furthermore, by elliptic regularity we know that
$u, v \in W^{4,2}_{loc}(E)$. Let $\{\xi_k\}$ be as in the hypothesis of this proposition. Using integration by parts (allowed since $\xi_ku\in  W^{4,2}_{comp}(E)$), we have
\begin{equation*}
(\xi_ku, \Delta_{B}^2v) = (\Delta_{B}(\xi_ku), \Delta_{B} v) =
((\Delta_g\xi_k)u, \Delta_{B} v) - 2(\nabla_{(d\xi_k)^\#}u, \Delta_{B} v)
+ (\xi_k\Delta_B u, \Delta_{B} v),
\end{equation*}
where we have used proposition \ref{adjoint_product_formula}(ii) to get the
second equality.

Analogously, we have
\begin{equation*}
(\Delta_{B}^2u, \xi_kv) = (\Delta_{B} u, \Delta_{B}(\xi_kv)) =
(\Delta_{B} u, (\Delta_g\xi_k)v) -2(\Delta_{B} u, \nabla_{(d\xi_k)^\#}v)
+ (\Delta_{B} u, \xi_k\Delta_{B} v).
\end{equation*}
Using the above two computations we then find
\begin{align}\label{main_thm_pf_eqn}
|(\xi_ku, L_{max}v) - (L_{max}u, \xi_kv)|
&\leq |((\Delta_g\xi_k)u, \Delta_{B} v) - (\Delta_{B} u, (\Delta_g\xi_k)v)| \nonumber\\
&\hspace{0.5cm}+ 2|(\nabla_{(d\xi_k)^\#}u, \Delta_{B} v) - (\Delta_{B} u,\nabla_{(d\xi_k)^\#}v)|,
\end{align}
which we will estimate term by term.

We start with
\begin{align*}
|(\nabla_{(d\xi_k)^\#}u, \Delta_{B} v)| &\leq
(|d\xi_k||\nabla u|, |\Delta_{B} v|)
\leq p_k(Q^{-3/4}|\nabla u|, |\Delta_{B} v|)
= p_k(Q^{-1/4}|\nabla u|, Q^{-1/2}|\Delta_{B} v|) \\
&\leq p_kJ_1[u]J_2[v],
\end{align*}
where to get the second inequality we used~(e5) with $\rho=3$,
and the last inequality follows from Cauchy--Schwarz inequality.

Similarly, we have
\begin{equation*}
|(\Delta_{B} u, \nabla_{(d\xi_k)^\#}v)| \leq p_kJ_1[v]J_2[u].
\end{equation*}

Next, using~(e5) with $\rho=3$ we obtain
\begin{align*}
|((\Delta_g\xi_k)u, \Delta_{B} v)| &\leq
(|\Delta_g\xi_k||u|, |\Delta_{B} v|)
\leq p_k
(Q^{-1/2}|u|, |\Delta_{B} v|)
= p_k(|u|, Q^{-1/2}|\Delta_{B} v|) \\
&\leq
p_kJ_0[u]J_2[v],
\end{align*}
where the last inequality follows from Cauchy--Schwarz inequality.

Similarly, we have
\begin{equation*}
|(\Delta_{B} u, (\Delta_g\xi_k)v)| \leq p_kJ_0[v]J_2[u].
\end{equation*}
Letting $k \rightarrow \infty$ in~\eqref{main_thm_pf_eqn} and using the dominated convergence theorem, the finiteness of $J_{j}[u]$ and $J_{j}[v]$ and the property $p_k\to 0$, we get
\begin{equation*}
|(u, L_{max}v) - (L_{max}u, v)| = 0,
\end{equation*}
which implies~(\ref{symm-max}). Hence, $L$ is essentially self-adjoint on on $C_c^{\infty}(E)$. \end{proof}

Then next proposition is concerned with lower semi-bounded operators.

\begin{prop}\label{prop-key-ests-lsb} Assume that the hypotheses of proposition~\ref{prop-key-ests} are satisfied with the following change: the sequence $\{\xi_k\}$ satisfies (e1)--(e3) and (e5) with $\rho=1$. In addition to $V\geq -Q$, assume that the condition~(\ref{E:semi-cond-1}) is satisfied. Then, $L$ is essentially self-adjoint on $C_{c}^{\infty}(E)$.
\end{prop}

Before proving the proposition, we recall the following definition, in which $A^*$ stands for the adjoint (in the operator theoretic sense) of the operator $A$ and $Im\, z$ stands for the imaginary part of $z\in\mathbb{C}$:
\begin{Def}\label{semimax_def}
Let $A$ be a symmetric operator in a Hilbert space $\mathscr{H}$ with inner product $(\cdot,\cdot)_{\mathscr{H}}$.
$A$ is called semimaximal if for all $u \in Dom(A^*)$ such that
$Im(A^*u, u)_{\mathscr{H}} = 0$, there exists a sequence $u_k \in Dom(A)$ such that
\begin{itemize}
\item[(i)] $u_k \rightarrow u$ in $\mathscr{H}$ and
\item[(ii)] $\displaystyle\lim_{k\rightarrow \infty}[(A^*u, u)_{\mathscr{H}} - (Au_k, u_k)_{\mathscr{H}} ]=0$.
\end{itemize}
\end{Def}

We will need the following abstract lemma about semimaximal operators;
see lemma 1 in \cite{brusentsev}.

\begin{lem}\label{abstract_semimax_lem}
Let $A$ be a symmetric operator in $\mathscr{H}$. Assume $A$ is semibounded
from below. Then the following are equivalent:
\begin{itemize}
\item[(i)] $A$ is essentially self-adjoint on $Dom(A)$,

\item[(ii)] $A$ is semimaximal.
\end{itemize}
\end{lem}

The following property will also be useful.

\begin{lem}\label{lem-min-max}
Let $u \in Dom(L_{min}^*) = Dom(L_{max})$. Then for all
$\phi \in W^{4,\infty}_{comp}(M)$ we have
$\phi u \in Dom(\overline{L_{min}})$.
\end{lem}
\begin{proof}
By elliptic regularity, we have that $u \in W^{4,2}_{loc}(E)$. As
$\phi \in W^{4,\infty}_{comp}(M)$, we find that
$\phi u \in W^{4,2}_{comp}(E)$. Hence, we can use
Friedrichs mollifiers to conclude that
$\phi u \in Dom(\overline{L_{min}})$.
\end{proof}
\begin{proof}[proof of proposition~\ref{prop-key-ests-lsb}]
Let $H := \overline{L_{min}}$ and observe that  $H^* = (\overline{L_{min}})^* = L_{min}^* = L_{max}$.
Let $u$ be an element of  $Dom(H^*)$ such that  $Im(H^*u, u) = 0$.
As $M$ is equipped with a sequence of functions $\{\chi_k\}$ satisfying (e1)--(e4), using propositions~\ref{J_2_finite} and~\ref{J_1_finite} we infer that $J_j[u]<\infty$ for $j=1,2$.

Let $\{\xi_k\}$ be as in the hypothesis of this proposition. Define $u_k:=\xi_k u$. Then, by lemma~\ref{lem-min-max} with $\phi=\xi_k$,
we have $u_k \in Dom(H)$.

Note that $u_k \rightarrow u$ in $L^2(E)$.
As $u_k \in Dom(H)$ we can use the same argument as in~(\ref{M-function_form_1}) to obtain
\begin{equation}\label{E:h-1-a}
(Hu_k,u_k)=(H(\xi_ku), \xi_ku) = Re(H(\xi_ku), \xi_ku) =
Re(\xi_k^2Lu, u) + P_{\xi_k}[u],
\end{equation}
where $P_{\xi_k}[u] = (\Delta_B(\xi_ku), \Delta_B(\xi_ku))
- Re(\Delta_Bu, \Delta_B(\xi_k^2u))$.

As $k \rightarrow \infty$, we have
$Re(\xi_k^2Lu, u) \rightarrow Re(Lu, u)$. Furthermore,
\begin{equation}\label{E:h-1-b}
Re(Lu, u) = Re(H^*u, u) = (H^*u, u),
\end{equation}
where the second equality follows from our assumption that
$Im(H^*u, u) = 0$.

Below we will show that $P_{\xi_k}[u] \rightarrow 0$ as
$k \rightarrow 0$. This, together with~(\ref{E:h-1-a}) and~(\ref{E:h-1-b}), leads to
\begin{equation*}
(Hu_k, u_k) - (H^*u, u) \rightarrow 0 \text{ as }
k\rightarrow \infty.
\end{equation*}
Therefore, $u_k$ satisfies the properties (i) and (ii) of definition~\ref{semimax_def}. Thus, $H$ is semimaximal. Next, observe that $H$ is semi-bounded from below and symmetric, as it is the closure of a semi-bounded from below and symmetric operator, namely
$L_{min}$. This means we can apply lemma \ref{abstract_semimax_lem} to infer
that $H$ is essentially self-adjoint, which means that $\overline{H}$ is self-adjoint. Noting that
$\overline{H} = \overline{\overline{L_{min}}} = \overline{L_{min}}$,
we see that $\overline{L_{min}}$ is self-adjoint. Therefore,
$L$, with domain $C_c^{\infty}(E)$, is essentially self-adjoint.

It remains to show that $P_{\xi_k}[u] \rightarrow 0$ as
$k \rightarrow 0$. In the same way as in lemma \ref{M-function_form_2}, with $\psi_k$ replaced by $\xi_k$, we obtain
\begin{equation*}
P_{\xi_k}[u] = 4||\nabla_{(d\xi_k)^\#}u||^2 + ||u\Delta_g\xi_k||^2
- 4Re(\nabla_{(d\xi_k)^\#}u, u\Delta_g\xi_k) +
2Re(|d\xi_k|_g^2u, \Delta_{B} u).
\end{equation*}

We estimate the first term as follows:
\begin{align*}
||\nabla_{(d\xi_k)^\#}u||^2 &\leq \int|d\xi_k|_g^2|\nabla u|^2\,d\mu
\leq p_k^2\int Q^{-1/2}|\nabla u|^2\,d\mu
= p_k^2(J_1[u])^2,
\end{align*}
where to get the second inequality we used~(e5) with $\rho=1$.

Next we estimate
\begin{align*}
||u\Delta_g\xi_k||^2 &= \int |u|^2|\Delta_g\xi_k|^2\,d\mu
\leq p_k^2\int |u|^2\,d\mu,
\end{align*}
where to bound the term $|\Delta_g\xi_k|^2$ we used~(e5) with $\rho=1$.

We now estimate
\begin{align*}
|(\nabla_{(d\xi_k)^\#}u, u\Delta_g\xi_k)| &\leq
\int |d\xi_k|_g|\nabla u||u||\Delta_g\xi_k|\,d\mu
\leq  p_k^2\int Q^{-1/4}|\nabla u||u|\,d\mu \\
&\leq  p_k^2\bigg{(}\int Q^{-1/2}|\nabla u|^2\,d\mu \bigg{)}^{1/2}
\bigg{(}\int |u|^2\,d\mu \bigg{)}^{1/2}
=  p_k^2J_1[u]J_0[u],
\end{align*}
where to get the second inequality we used~(e5) with $\rho=1$, and
the third inequality follows from Cauchy--Schwarz inequality.

Lastly, we estimate
\begin{align*}
|(|d\xi_k|_g^2u, \Delta_{B} u)| &\leq \int |u||d\xi_k|^2_g|\Delta_{B} u|\,d\mu
\leq  p_k^2\int |u|Q^{-1/2}|\Delta_{B} u|\,d\mu \\
&\leq  p_k^2\bigg{(}\int |u|^2\,d\mu \bigg{)}^{1/2}
\bigg{(}\int Q^{-1}|\Delta_{B} u|^2\,d\mu \bigg{)}^{1/2}
=  p_k^2J_0[u]J_2[u],
\end{align*}
where to bound the term $|d\xi_k|_g$ we used~(e5) with $\rho=1$, and
the third inequality follows from Cauchy--Schwarz inequality.

Letting $k\rightarrow 0$ in the above four estimates and remembering that $\displaystyle \lim_{k\to\infty} p_k=0$, it follows that
$P_{\xi_k}[u] \rightarrow 0$ as $k \rightarrow \infty$.
\end{proof}

\section{Proofs of theorems \ref{main_thm_1} and \ref{main_thm_2}}\label{proof_thm_1}
\subsection{Smooth cut-off functions}\label{smooth-cut-off}
Here we recall that if $(M, g)$  a complete Riemannian manifold with $Ric_{g}\geq -K_1$, where $K_1\geq 0$ is a constant, then there exists a sequence
of functions $\chi_k \in C^{\infty}_c(M)$, indexed by $k \in \NN$,
with the following properties:
\begin{itemize}
\item[1.] For all $x \in M$ and all $k \in \NN$, we have
$0 \leq \chi_k(x) \leq 1$.

\item[2.] There exists $\gamma > 1$ such that for all $k \in \NN$ we
have $\chi_k = 1$ on $B_k(x_0)$ and $supp \chi_k \subseteq
B_{\gamma k}(x_0)$.

\item[3.] $\sup_{x \in M}|d\chi_k(x)|_{g} \leq \frac{C}{k}$, where
$C > 0$ is a constant independent of $k$.

\item[4.] $\sup_{x \in M}|\Delta_{g}\chi_k(x)| \leq \frac{C}{k}$,
where
$C > 0$ is a constant independent of $k$.
\end{itemize}

Note that the sequence $\{\chi_k\}$ satisfies the properties (e1)--(e3) of section~\ref{cut-off_section}.

\begin{rem}\label{gb}
Under the assumptions stated above, the existence of a sequence $\{\chi_k\}$ with properties (1)--(4) was proved in theorem III.3(b) of~\cite{guneysu-17}. Under more general assumptions, informally described in remark~\ref{rem-cut-off} above, the existence of such functions was established subsequently in corollary 2.3 of \cite{bianchi}. In the literature, the functions $\chi_k\in C^{\infty}_c(M)$ with properties (1)--(4) are called  \emph{Laplacian cut-off functions}. For further references on the question of existence of a sequence of (weak) Laplacian/Hessian cut-off functions and a related problem concerning the existence of suitable exhaustion functions on a complete Riemannian manifold (with additional geometric assumptions), see~\cite{huang, irv-19, rv-19}.
\end{rem}

With the cut-off functions $\{\chi_k\}$  in place, we define
\begin{equation}\label{E:psi-k-def}
\psi_k:= Q^{-1/2}\chi_k^4,
\end{equation}
where $Q\colon M\to[1,\infty)$ is a function belonging to $C^{\infty}(M)$ and satisfying the hypotheses (i) and (ii) of theorem \ref{main_thm_1}.

\subsection{Additional estimates}\label{add-est}
The proof of theorem~\ref{main_thm_1} rests on the next two lemmas. In the first one, we show that $\{\chi_k\}$ satisfies (e4) of section~\ref{cut-off_section}.

\begin{lem}\label{l-s-1} Let $(M,g)$ be a complete Riemannian manifold with with $Ric_{g}\geq -K_1$, where $K_1\geq 0$ is a constant. Assume that $Q\colon M\to[1,\infty)$ belongs to $C^{\infty}(M)$ and satisfies the hypotheses (i) and (ii) of theorem~\ref{main_thm_1}. Let $\psi_k$ be as in~(\ref{E:psi-k-def}). Then,
\begin{equation*}
|d\psi_k|_{g} \leq CQ^{-1/4}\chi_k^3 \text{ and } |\Delta_{g}\psi_k| \leq C\chi_k^2,
\end{equation*}
where $C>0$ is a constant independent of $k$.
\end{lem}
\begin{proof} To make the notations simpler, the symbol $|\cdot|$ refers to the norm with respect to the metric induced on $T^*M$ from $g$. We have
\begin{align*}
|d\psi_k| &= |d(Q^{-1/2}\chi_k^4)|
= |d(Q^{-1/2})\chi_k^4 + 4Q^{-1/2}\chi_k^3d\chi_k|
\leq \frac{1}{2}|Q^{-3/2}\chi_k^4dQ| + 4|Q^{-1/2}\chi_k^3d\chi_k| \\
&\leq CQ^{-1/4}\chi_k^4 + CQ^{-1/2}\chi_k^3
\leq CQ^{-1/4}\chi_k^3
\end{align*}
where to get the second inequality we have used the assumption 
$|dQ| \leq CQ^{5/4}$ and
the property $|d\chi_k| \leq \frac{C}{k}$.

Using proposition~\ref{Laplace_chain_rule}(i) and proposition~\ref{adjoint_product_formula}(ii), we have
\begin{align*}
\Delta_{g}\psi_k &= \Delta_{g}(Q^{-1/2}\chi_k^4)
=\chi_k^4\Delta_{g}(Q^{-1/2}) - 2\langle d\chi_k^4, dQ^{-1/2}\rangle +
Q^{-1/2}\Delta_{g}(\chi_k^4) \\
&= \chi_k^4\Delta_{g}(Q^{-1/2}) - 2\langle d\chi_k^4, dQ^{-1/2}\rangle +
4Q^{-1/2}\chi_k^3\Delta_g\chi_k - 12\chi_k^2Q^{-1/2}|d\chi_k|^2 \\
&=-\frac{3}{4}\chi_k^4Q^{-5/2}|dQ|^2 - \frac{1}{2}\chi_k^4Q^{-3/2}
\Delta_{g} Q - 2\langle d\chi_k^4, dQ^{-1/2}\rangle +
4Q^{-1/2}\chi_k^3\Delta_{g}\chi_k \\
&\hspace{0.5cm}- 12\chi_k^2Q^{-1/2}|d\chi_k|^2 \\
&=-\frac{3}{4}\chi_k^4Q^{-5/2}|dQ|^2 -
\frac{1}{2}\chi_k^4Q^{-3/2}\Delta_{g} Q
+4\chi_k^3Q^{-3/2}\langle d\chi_k, dQ\rangle +
4Q^{-1/2}\chi_k^3\Delta_{g}\chi_k \\
&\hspace{0.5cm}- 12\chi_k^2Q^{-1/2}|d\chi_k|^2.
\end{align*}
We then estimate
\begin{align*}
|\Delta_{g}\psi_k| &\leq \frac{3}{4}\chi_k^4Q^{-5/2}|dQ|^2 +
\frac{1}{2}\chi_k^4Q^{-3/2}|\Delta_g Q| +
 4\chi_k^3Q^{-3/2}|d\chi_k||dQ| + 4\chi_k^3Q^{-1/2}|\Delta_g\chi_k| \\
&\hspace{0.5cm} + 12\chi_k^2Q^{-1/2}|d\chi_k|^2 \\
 &\leq
 C\chi_k^4Q^{-5/2}Q^{5/2} + C\chi_k^4Q^{-3/2}Q^{3/2} +
 C\chi_k^3Q^{-3/2}Q^{5/4} + C\chi_k^3Q^{-1/2} + C\chi_k^2Q^{-1/2} \\
 &\leq C\chi_k^2,
\end{align*}
where we have used the bounds on $|dQ|$, $|\Delta_g Q|$,
$|d\chi_k|$, $|\Delta_g\chi_k|$, the fact that $\chi_k \leq 1$, and
that $Q \geq 1$.
\end{proof}

\begin{rem} \label{R: metric-tilde-g} Consider the metric $\tilde{g}:=Q^{-\rho/2}g$, where $\rho\geq 1$. If the metric $\tilde{g}$ is complete and $Ric_{\tilde{g}}$ is bonded from below, then there exists a sequence $\xi_k \in C^{\infty}_c(M)$ satisfying the properties (1)--(4) of section~\ref{smooth-cut-off}. In particular, for the metric $\tilde{g}$, the properties (3) and (4) read as follows: $|d\xi_k|_{\tilde{g}}\leq C/k$ and $|\Delta_{\tilde{g}}\xi_k|\leq C/k$.
\end{rem}

The next lemma provides the estimates for $|d\xi_k|_{{g}}$ and $|\Delta_{{g}}\xi_k|$.

\begin{lem}\label{l-s-2} Let $(M,g)$ be a Riemannian manifold. Let $Q\colon M\to[1,\infty)$ be a function in $C^{\infty}(M)$ satisfying the hypotheses (i) and (ii) of theorem~\ref{main_thm_1}. Let $\tilde{g} := Q^{-\rho/2}g$, where $\rho\geq 1$.  Assume that the metric $\tilde{g}$ is complete and $Ric_{\tilde{g}} \geq -K_2$ for some constant $K_2 \geq 0$.  Let $\{\xi_k\}$ be as in remark~\ref{R: metric-tilde-g}. Then
\begin{equation*}
|d\xi_k|_g \leq \frac{C}{k}Q^{-{\rho}/4}\textrm{ and }|\Delta_g\xi_k| \leq  \frac{C}{k}Q^{-(\rho-1)/4},
\end{equation*}
where $C>0$ is a constant independent of $k$.
\end{lem}
\begin{proof}
By the definition of $\tilde{g}$ we have
\[
|d\xi_k|_g = Q^{-\rho/4}|d\xi_k|_{\tilde{g}}\leq \frac{C}{k}Q^{-\rho/4},
\]
where the inequality comes from the property $|d\xi_k|_{\tilde{g}}\leq C/k$.
Using corollary \ref{Lapl_conf_trans_bound} and the assumption $|d Q|_{g}\leq C Q^{5/4}$, we estimate the term $|\Delta_g\xi_k|$ as follows:
\begin{align}
|\Delta_g\xi_k| &\leq Q^{-\rho/2}\frac{C}{k} + \left|\frac{(n-2)\rho}{4}\right|
Q^{-(\rho+4)/4}|dQ|_{g}\frac{C}{k}
\leq Q^{-\rho/2}\frac{C}{k}+\left|\frac{(n-2)\rho}{4}\right|Q^{-(\rho+4)/4}Q^{5/4}\frac{C}{k}\nonumber\\
&\leq \frac{C}{k}Q^{-(\rho-1)/4}.
\end{align}
\end{proof}
\begin{proof}[proof of theorem \ref{main_thm_1}] In this section we assume that all hypotheses of theorem~\ref{main_thm_1} are satisfied. As discussed in section~\ref{smooth-cut-off}, since the metric $g$ is complete (see remark~\ref{comment-compl}) and $Ric_{g}\geq -K_1$, there exists a sequence $\{\chi_k\}$ in $C_{c}^{\infty}(M)$ satisfying the properties (e1)--(e3) of section~\ref{cut-off_section}. Additionally, by lemma~\ref{l-s-1} the sequence $\{\chi_k\}$ satisfies the property (e4). As discussed in remark~\ref{R: metric-tilde-g}, in view of the completeness of $\tilde{g}:=Q^{-3/2}g$ and the assumption $Ric_{\tilde{g}}\geq -K_2$, there exists a sequence $\{\xi_k\}$ in  $C_{c}^{\infty}(M)$ satisfying the properties (e1)--(e3). Furthermore, we can use lemma~\ref{l-s-2} with $\rho=3$ to infer that the sequence $\{\xi_k\}$ satisfies the property (e5) with $\rho=3$. Thus, all hypotheses of proposition~\ref{prop-key-ests} are satisfied. Therefore, $L$ is essentially self-adjoint on $C_{c}^{\infty}(E)$.
\end{proof}

\begin{proof}[proof of theorem \ref{main_thm_2}]
In this section we assume that all hypotheses of theorem~\ref{main_thm_2} are satisfied. The proof is the same as that of theorem~\ref{main_thm_1} with the following change: in view of the completeness of $\tilde{g}:=Q^{-1/2}g$, we now use remark~\ref{R: metric-tilde-g} and lemma~\ref{l-s-2} with $\rho=1$. In addition to the assumption $V(x)\geq -Q(x)$, we assume that~(\ref{E:semi-cond-1}) is satisfied. This enables us to use proposition~\ref{prop-key-ests-lsb} to infer that the operator $L$ is essentially self-adjoint on $C_{c}^{\infty}(E)$.
\end{proof}

\section{Proofs of corollaries~\ref{cor1_1} and~\ref{cor1_3}}\label{S:cor-1}

\begin{proof}[proof of corollary \ref{cor1_1}] By proposition \ref{Ric_bounded_conform_trans}, it follows that $Ric_{\tilde{g}}\geq-K$, where $K\geq 0$ is some constant.
Keeping in mind proposition~\ref{Laplace_chain_rule}(iii), note that $|dQ|_g \leq CQ^{1/4}$ and $|Hess_{g}(Q)|_g \leq CQ^{-1/2}$ ensure the fulfillment of assumptions (i) and (ii) of theorem~\ref{main_thm_1}. Thus, $L$ is essentially self-adjoint on $C_c^{\infty}(E)$.
\end{proof}

\begin{proof}[proof of corollary \ref{cor1_3}] By proposition \ref{Ric_bounded_conform_trans-2}, it follows that $Ric_{\tilde{g}}\geq-K$, where $K\geq 0$ is some constant. Keeping in mind proposition~\ref{Laplace_chain_rule}(iii), note that $|dQ|_g \leq CQ^{3/4}$ and $|Hess_{g}(Q)|_g \leq CQ^{1/2}$ ensure the fulfillment of assumptions (i) and (ii) of theorem~\ref{main_thm_2}. Thus, $L$ is essentially self-adjoint on $C_c^{\infty}(E)$.
\end{proof}

\section{Proofs of theorems \ref{main_thm_1-s} and \ref{main_thm_2-s}}\label{proof_thm_1-s}
\subsection{Cut-off functions (bounded sectional curvature)}\label{cut-off-bsc}
If $(M,g)$ is complete and $|Sec_{g}|_{g}$ is bounded, according to a theorem in~\cite{shi} there exists a smooth function $\beta : M \rightarrow [1, \infty)$ and a constant $\widehat{C}$ such that
\begin{equation}\label{beta-r}
r_{g}(x) + 1 \leq \beta(x) \leq \widehat{C} + r_{g}(x),\quad |d\beta|_g \leq \widehat{C}, \textrm{ and }|Hess_{g}(\beta)|_{g}\leq \widehat{C},
\end{equation}
for all $x\in M$, where $r_{g}$ is as in~(\ref{riem-dist}). In particular, we have
\begin{equation}\label{beta-r-1}
0\leq r_{g}(x)\leq \beta(x),
\end{equation}
and from proposition~\ref{Laplace_chain_rule}(iii), we see that there exists a constant $C$ such that
\begin{equation}\label{beta-r-2}
|d\beta|_g \leq C\textrm{ and }|\Delta_{g}\beta|\leq C,
\end{equation}
for all $x\in M$. 

Let $f\colon [0,\infty) \rightarrow [1, \infty)$ be a function satisfying the assumptions (i), (iii) and (iv) of theorem~\ref{main_thm_1-s} and the property
\begin{equation}\label{div-kappa}
\int_{0}^{\infty}f^{-\rho/4}(t)\,dt=\infty,
\end{equation}
where $\rho\geq 1$.
\begin{rem} In theorems \ref{main_thm_1-s} and \ref{main_thm_2-s}, we use the condition~(\ref{div-kappa}) for $\rho=3$ and $\rho=1$ respectively. Note that if~(\ref{div-kappa}) is satisfied for $\rho=3$, then it will be satisfied for $\rho=1$.
\end{rem}
With $\beta$ as in~(\ref{beta-r}) and $\rho$ as in~(\ref{div-kappa}), we define $P\colon M\to[0,\infty)$ as follows:
\begin{equation}\label{P-def-s}
P(x):=\int_{0}^{\beta(x)}f^{-\rho/4}(t)\,dt.
\end{equation}
Denoting by $s_{+}(x):=\max\{s(x),0\}$ the positive part of a function $s\colon M\to\mathbb{R}$, for $k\in\mathbb{N}$ we define
\begin{equation}\label{chi-k-def-s}
\chi_k(x):=\left(1-\frac{P(x)}{k}\right)_{+}.
\end{equation}
\begin{rem} To keep the notations simpler, the symbols $P$ and $\chi_k$ will not explicitly indicate the dependence of these functions on $\rho$. The same notational simplification applies to the functions  $\psi_k$ and $\xi_k$ introduced below.
\end{rem}
Note that the sequence $\{\chi_k\}$ satisfies the properties (e1) and (e3) of section~\ref{cut-off_section}.
Using the completeness of $(M,g)$, the inequality~(\ref{beta-r-1}), and the condition~(\ref{div-kappa}) with $\rho\geq 1$, we see that the functions $\chi_k$ are compactly supported. Moreover, since $f$ and $\beta$ are smooth, we see that (e2) is satisfied.

Define $Q\colon M \to [1,\infty)$ by the formula
\begin{equation}\label{Q-def-s}
Q(x):=(f\circ \beta)(x),
\end{equation}
and let
\begin{equation}\label{psi-k-def-s}
\psi_k(x):=(\chi_k(x))^4Q^{-1/2}(x).
\end{equation}

\subsection{More estimates (bounded sectional curvature)}\label{est-bsc}
We start this section with an observation about $Q$.
\begin{lem}\label{p-Q-f}
If $f\colon [0,\infty)\to [1,\infty)$ is smooth and satisfies the hypotheses (iii) and (iv) of theorem~\ref{main_thm_1-s}, then $Q$ satisfies  the hypotheses (i) and (ii) of theorem~\ref{main_thm_1}.
\end{lem}
\begin{proof}
To prove that $Q$ satisfies the hypothesis (i) of theorem~\ref{main_thm_1},
note that $dQ = f'(\beta)d\beta$.  The result then follows by using the first estimate in~(\ref{beta-r-2}) together with the fact that $f$ satisfies the hypothesis (iii) of theorem~\ref{main_thm_1-s}.

To prove that $Q$ satisfies the hypothesis (ii) of theorem~\ref{main_thm_1}, we
use proposition~\ref{Laplace_chain_rule}(i) to obtain
\begin{equation*}
\Delta_gQ = -f''(\beta)|d\beta|_g^2 + f'(\beta)\Delta_g\beta.
\end{equation*}
The result then follows by using the estimates~(\ref{beta-r-2}) along with the assumption that $f$ satisfies the hypotheses (iii) and (iv)
of theorem~\ref{main_thm_1-s}.
\end{proof}

Now we list some properties of $P$.
\begin{lem}\label{P-k-suppl} Let $(M,g)$ be a complete Riemannian manifold with $|Sec_{g}|_{g}\leq K_3$, where $K_3$ is a constant. Assume that $f\colon [0,\infty)\to[1,\infty)$ satisfies the hypotheses (i), (iii) and (iv) of theorem~\ref{main_thm_1-s} and the condition~(\ref{div-kappa}) with $\rho\geq 1$. Let $P$ and $Q$ be as in~(\ref{P-def-s}) and ~(\ref{Q-def-s}). Then, for all $x\in M$ we have
\begin{itemize}
\item[(i)] $P(x)\geq f^{-\rho/4}(1)$,
\item[(ii)] $|dP(x)|_{g}\leq CQ^{-\rho/4}(x)$,
\item[(iii)] $|\Delta_{g}P(x)|\leq CQ^{-(\rho-1)/4}(x)$.
\end{itemize}
where $C>0$ is a constant.
\end{lem}
\begin{proof} To show part (i), we look at~(\ref{beta-r}) and observe that $\beta(x)\geq 1$, for all $x\in M$. Therefore,
\[
P(x)=\int_{0}^{\beta(x)}f^{-\rho/4}(t)\,dt\geq\int_{0}^{1}f^{-\rho/4}(t)\,dt\geq f^{-\rho/4}(1),
\]
where the second inequality follows since $f$ is non-decreasing.

Part (ii) follows from the fundamental theorem of calculus, the definition~(\ref{Q-def-s}), and the first estimate in \eqref{beta-r-2}:
\[
|dP(x)|_g = |f^{-\rho/4}(\beta(x))d\beta(x)|_g \leq CQ^{-\rho/4}.
\]

Finally, to prove part (iii), we apply (i) of proposition \ref{adjoint_product_formula}, estimates \eqref{beta-r-2}, and the hypothesis (iii) of theorem~\ref{main_thm_1-s}:
\begin{align*}
|\Delta_gP(x)| &= |d^{\dagger}dP(x)| \\
&= |d^{\dagger}(f^{-\rho/4}(\beta(x))d\beta(x))| \\
&= |f^{-\rho/4}(\beta(x))\Delta_g\beta(x) - \frac{\rho}{4}
\langle f^{-(\rho+4)/4}(\beta(x))f'(\beta(x))d\beta(x), d\beta(x)
\rangle| \\
&\leq CQ^{-(\rho-1)/4}(x),
\end{align*}
where $d^{\dagger}$ is the formal adjoint of $d$ and $\langle\cdot,\cdot\rangle$ is the (fiberwise) scalar product in $\Lambda^1T^*_{x}M$.

\end{proof}

\begin{lem}\label{psi-k-suppl} Assume that the hypotheses of lemma~\ref{P-k-suppl} are satisfied. Let $\psi_k$ be as in~(\ref{psi-k-def-s}). Then, we have
\begin{equation*}
|d\psi_k|_{g} \leq CQ^{-1/4}\chi_k^3 \text{ and } |\Delta_{g}\psi_k| \leq C\chi_k^2,
\end{equation*}
where $C>0$ is a constant independent of $k$.
\end{lem}
\begin{proof} Keeping in mind lemma~\ref{p-Q-f} and remembering that $\rho\geq 1$, the proof is carried out in the same way as that of lemma~\ref{l-s-1} with the following change: instead of estimates (3) and (4) of section~\ref{smooth-cut-off}, we use the estimates
\begin{equation*}
|d\chi_k|_{g} \leq |dP|_{g}\leq CQ^{-\rho/4}\text{ and } |\Delta_{g}\chi_k| \leq |\Delta_{g}P(x)|\leq CQ^{-(\rho-1)/4},
\end{equation*}
which follow from the definition of $\chi_k$ in~(\ref{chi-k-def-s}) and lemma~\ref{P-k-suppl}.
\end{proof}
For $k\in\NN$, $k\geq 2$, and $x\in M$ define
\begin{equation}\label{xi-k-def-s}
\xi_k(x):=\left(1-\left(\frac{P(x)}{k}\right)^{1/\sqrt{\ln k}}\right)^4_{+}.
\end{equation}
For the same reasons as in the case of $\{\chi_k\}$, the sequence $\{\xi_k\}$ satisfies (e1)--(e3). The next lemma shows that the sequence $\{\xi_k\}$ satisfies (e5).

\begin{lem}\label{xi-k-suppl} Assume that the hypotheses of lemma~\ref{P-k-suppl} are satisfied. Let $\xi_k$ be as in~(\ref{xi-k-def-s}). Then,
\begin{equation*}
|d \xi_k|_{g} \leq \frac{C}{\sqrt{\ln k}}Q^{-\rho/4}\text{ and } |\Delta_{g}\xi_k| \leq \frac{C}{\sqrt{\ln k}}Q^{-(\rho-1)/4}
\end{equation*}
where $C>0$ is a constant independent of $k$.
\end{lem}
\begin{proof}
The estimate for $|d \xi_k|_{g}$ follows from the
chain rule and (ii) of lemma \ref{P-k-suppl}.

The estimate for $|\Delta_{g}\xi_k|$ follows by applying 
proposition \ref{Laplace_chain_rule}(i) together with the properties (ii) and (iii) of lemma
\ref{P-k-suppl}.
\end{proof}

%Moreover, from the hypotheses (iii) and (iv) on $f$ and~({beta-r-2}), we see imply that $Q=f\circ \beta$ satisfies the conditions (i) and (ii) of theorem~\ref

\begin{proof}[proof of theorem \ref{main_thm_1-s}] Here we assume that all hypotheses of theorem~\ref{main_thm_1-s} are satisfied. In particular, the hypothesis (ii) is the same as the condition~(\ref{div-kappa}) with $\rho=3$. As indicated in section~\ref{cut-off-bsc} and in lemma~\ref{psi-k-suppl}, the sequence $\{\chi_k\}$ in~(\ref{chi-k-def-s}) satisfies the properties (e1)--(e4) of section~\ref{cut-off_section}. Furthermore, the sequence $\{\xi_k\}$ in~(\ref{xi-k-def-s}) satisfies the properties (e1)--(e3) and, by lemma~\ref{xi-k-suppl} with $\rho=3$, the same sequence satisfies (e5) with $\rho=3$. Remembering~(\ref{beta-r-1}) and the non-decreasing property of $f$, the assumption
$V(x)\geq -(f\circ r_{g})(x)$ implies $V(x) \geq -(f\circ \beta)(x)=-Q(x)$, where the last equality follows from~(\ref{Q-def-s}). Thus, the hypotheses of proposition~\ref{prop-key-ests} are satisfied. Therefore, $L$ is essentially self-adjoint on $C_{c}^{\infty}(E)$.
\end{proof}

\begin{proof}[proof of theorem \ref{main_thm_2-s}] Here we assume that all hypotheses of theorem~\ref{main_thm_2-s} are satisfied. In particular, the hypothesis (ii) is the same as the condition~(\ref{div-kappa}) with $\rho=1$. The proof of theorem \ref{main_thm_2-s} is the same as that of theorem~\ref{main_thm_1-s} with the following change: we use lemma~\ref{xi-k-suppl} with $\rho=1$ to infer that the sequence $\{\xi_k\}$ satisfies~(e5) with $\rho=1$. In addition to the assumption $V(x)\geq -(f\circ r_{g})(x)$, we use the lower semi-boundedness condition~(\ref{E:semi-cond-1}). Thus, the hypotheses of proposition~\ref{prop-key-ests-lsb} are satisfied. Hence, $L$ is essentially self-adjoint on $C_{c}^{\infty}(E)$.
\end{proof}

\begin{appendices} \section{Appendix}

In this section we gather together
various results pertaining to connections and conformal changes of metrics.

The first two formulas of the following proposition describe the chain rules for the Laplacian/Hessian.
For the first formula, see exercises 3.4 and 3.9 in \cite{grigoryan}. We remind the reader that in our paper the Laplace-Beltrami operator $\Delta_{g}$ is non-negative, which explains the sign difference with the corresponding formulas of~\cite{grigoryan}. The second formula in proposition~\ref{Laplace_chain_rule} is obtained by combining the formula for the differential of a composition (exercise 3.4 of \cite{grigoryan}) with the definition of Hessian $Hess_{g}(f):=\nabla^{lc,g}df$. (Here, $\nabla^{lc,g}$ is the covariant derivative on $T^*M$ induced from the Levi--Civita connection on $(M,g)$ and $d$ is the standard differential.) For the third formula of proposition~\ref{Laplace_chain_rule}, we refer the reader to (III.24) in~\cite{guneysu-17}.

\begin{prop}\label{Laplace_chain_rule}
Let $(M, g)$ be a an $n$-dimensional Riemannian manifold, with Riemannian metric $g$.
Let $v : M \rightarrow \RR$ be a smooth function, and
$f : U \rightarrow \RR$ a smooth function, where $U$ is an open set
of $\RR$ containing the range of $v$. We then have:
\begin{enumerate}
\item [(i)] $\Delta_g(f \circ v) = -f''(v)|dv|_g^2 +f'(v)\Delta_gv$,
\item [(ii)] $Hess_g(f \circ v) = f''(v)dv\otimes dv + f'(v)Hess_g(v)$,
\item[(iii)] $|\Delta_gf|\leq \sqrt{n}|Hess_g(f)|_{g}$, where the inequality is understood in pointwise sense.
\end{enumerate}

\end{prop}

Next we recall product rules for the formal adjoint of a connection and
for the Bochner Laplacian.

\begin{prop}\label{adjoint_product_formula}
Let $(M, g)$ be a Riemannian manifold and let $E$ be a Hermitian vector bundle over $M$  with
a Hermitian connection $\nabla$. Let $\Delta_{B} := \nabla^{\dagger}\nabla$ denote
the associated Bochner Laplacian. Let $f \in  W^{2,\infty}_{loc}(M)$ and
$u \in W^{2,2}_{loc}(E)$. Then
\begin{enumerate}
\item [(i)]$\nabla^{\dagger}(f\nabla u) = f\Delta_{B} u - \nabla_{(df)^\#}u$,
\item [(ii)] $\Delta_{B}(fu) = f\Delta_{B} u - 2\nabla_{(df)^\#}u + u\Delta_g f$,
\end{enumerate}
where $(df)^{\#}$ stands for the vector field corresponding to $df$ via the metric $g$.
\end{prop}
\begin{proof} Using integration by parts and the product rule for $\nabla$, for all $v\in C_{c}^{\infty}(E)$ we have
\begin{align}
(\nabla^{\dagger}(f\nabla u),v)&=(\nabla u,f\nabla v)=(\nabla u, \nabla(fv))-(\nabla u, df\otimes v)\nonumber\\
&=(\nabla^{\dagger}\nabla u, fv)-(\nabla_{(df)^\#} u,v)= (f\nabla^{\dagger}\nabla u, v)-(\nabla_{(df)^\#} u,v),\nonumber
\end{align}
which gives the formula (i). The formula (ii) will then follow by using the product rule for $\nabla$, the formula (i) of this proposition, and the following formula (see the equation~(III.7) of~\cite{guneysu-17}):
\begin{equation*}
\nabla^{\dagger}(\omega\otimes z)= (d^{\dagger}\omega)z-\nabla_{\omega^\#}z,
\end{equation*}
where $z\in W^{1,2}_{loc}(E)$ and $\omega$ is a $1$-form on $M$ belonging to $W^{1,\infty}_{loc}(\Lambda^1T^*M)$.
\end{proof}

\begin{prop}\label{d_conf_trans_est}
Let $(M, g)$ be a Riemannian manifold, with Riemannian metric $g$.
Let $\tilde{g} = \lambda^{\alpha} g$, where
$\lambda : M \rightarrow (0, \infty)$ is a smooth function, and
$\alpha \in \RR$.
Let $f : M \rightarrow \CC$ be a $C^1$ function on $M$. Suppose we have
the bound $|df|_{\tilde{g}} \leq \phi$, where
$\phi : M \rightarrow [0, \infty)$.
Then we have
\begin{equation*}
|df|_{g} \leq \lambda^{\alpha/2}\phi.
\end{equation*}
\end{prop}
\begin{proof}
This is a pointwise estimate, so it suffices to work in coordinates.
Let $(x^i)$ be local coordinates about a point in $M$. We then
compute
\begin{equation*}
\langle df , df\rangle_{g} =
g^{ij}\frac{\partial{f}}{\partial{x}^i}
\frac{\partial{f}}{\partial{x}^j} =
\lambda^{\alpha}\tilde{g}^{ij}\frac{\partial{f}}{\partial{x}^i}
\frac{\partial{f}}{\partial{x}^j} =
\lambda^{\alpha}\langle df, df\rangle_{\tilde{g}}.
\end{equation*}
The result then follows.
\end{proof}

\begin{prop}\label{Lapl_conf_trans_form}
Let $(M, g)$ be a Riemannian manifold, with Riemannian metric $g$.
Let
$\tilde{g} = \lambda^{\alpha} g$, where
$\lambda : M \rightarrow (0, \infty)$ is a smooth function, and
$\alpha \in \RR$.
Let $P : M \rightarrow \CC$ be a $C^2$ function on $M$.
We then have the following formula for $\Delta_gP$ in terms of
$\Delta_{\tilde{g}}P$
\begin{equation*}
\Delta_gP = \lambda^{\alpha}\Delta_{\tilde{g}}P +
\bigg{(}\frac{2\alpha -n\alpha}{2}\bigg{)}\lambda^{(\alpha -1)}
\langle dP, d\lambda\rangle_{\tilde{g}}.
\end{equation*}
\end{prop}
\begin{proof}
In local coordinates $(x^i)$ we can write
$\Delta_gP =
\frac{1}{\rho}\frac{\partial}{\partial{x^j}}\bigg{(}
\rho g^{jk}\frac{\partial{P}}{\partial{x^k}} \bigg{)}$, where
$\rho = \sqrt{det(g_{ij})}$ with $g_{ij}$ being the components of the
metric tensor in $(x^i)$ coordinates. It is then easy to see
that $\tilde{\rho} = \sqrt{det(\tilde{g}_{ij})} =
\lambda^{n\alpha/2}\rho$.

In these local coordinates, we compute
\begin{align*}
\Delta_gP &=
\frac{1}{\rho}\frac{\partial}{\partial{x^j}}\bigg{(}
\rho g^{jk}\frac{\partial{P}}{\partial{x^k}} \bigg{)} \\
&= \frac{1}{\lambda^{-n\alpha/2}\tilde{\rho}}
\frac{\partial}{\partial{x^j}}\bigg{(}
\lambda^{-n\alpha/2}\tilde{\rho}\lambda^{\alpha}\tilde{g}^{jk}
\frac{\partial{P}}{\partial{x}^k}\bigg{)} \\
&= \lambda^{\alpha}\Delta_{\tilde{g}}P +
\frac{1}{\lambda^{-n\alpha/2}\tilde{\rho}}\tilde{\rho}
\tilde{g}^{jk}\frac{\partial{P}}{\partial{x^k}}
\frac{\partial}{\partial{x^j}}\bigg{(}
\lambda^{\frac{-n\alpha + 2\alpha}{2}}
\bigg{)} \\
&=  \lambda^{\alpha}\Delta_{\tilde{g}}P + \bigg{(}
\frac{2\alpha - n\alpha}{2} \bigg{)}
\lambda^{(\alpha - 1)}\tilde{g}^{jk}\frac{\partial{P}}{\partial{x}^k}
\frac{\partial{\lambda}}{\partial{x}^j} \\
&=  \lambda^{\alpha}\Delta_{\tilde{g}}P +
\bigg{(}
\frac{2\alpha - n\alpha}{2} \bigg{)}\lambda^{(\alpha - 1)}
\langle dP, d\lambda\rangle_{\tilde{g}}.
\end{align*}
\end{proof}

\begin{cor}\label{Lapl_conf_trans_bound}
Assume that the hypotheses of proposition \ref{Lapl_conf_trans_form} are satisfied.
Suppose we have the bound
$|dP|_{\tilde{g}} \leq h_1$ and
$|\Delta_{\tilde{g}}P| \leq h_2$, where
$h_i : M \rightarrow [0, \infty)$ for $i = 1, 2$. Then
\begin{equation*}
|\Delta_gP|\leq \lambda^{\alpha}h_2 +
\bigg{|}\frac{2\alpha - n\alpha}{2} \bigg{|}
\lambda^{\frac{\alpha - 2}{2}}|d\lambda|_gh_1.
\end{equation*}
\end{cor}
\begin{proof}
Using the formula $\langle dP, d\lambda\rangle_{\tilde{g}}=\lambda^{-\alpha}\langle dP, d\lambda\rangle_{g}$, we write proposition \ref{Lapl_conf_trans_form} in a slightly different
form:
\begin{equation*}
\Delta_gP = \lambda^{\alpha}\Delta_{\tilde{g}}P +
\bigg{(}\frac{2\alpha -n\alpha}{2}\bigg{)}\lambda^{-1}
\langle dP, d\lambda\rangle_{g}.
\end{equation*}
Using the above formula, we estimate
\begin{align*}
|\Delta_gP| &\leq \lambda^{\alpha}|\Delta_{\tilde{g}}P| +
\bigg{|}\frac{2\alpha - n\alpha}{2} \bigg{|}\lambda^{-1}
|dP|_g|d\lambda|_g \\
&= \lambda^{\alpha}|\Delta_{\tilde{g}}P| +
\bigg{|}\frac{2\alpha - n\alpha}{2} \bigg{|}\lambda^{-1}
|dP|_g|d\lambda|_g \\
&\leq
\lambda^{\alpha}h_2 +
\bigg{|}\frac{2\alpha - n\alpha}{2} \bigg{|}\lambda^{-1}
|dP|_g|d\lambda|_g.
\end{align*}
We can then estimate the $|dP|_g$ term using proposition
\ref{d_conf_trans_est}, and the corollary follows.
\end{proof}

We will need a formula that tells us how the
Ricci curvature changes under a conformal transformation.
For a proof of this proposition, the reader can consult \cite{besse}.

\begin{prop}\label{Ricci_curvature_conform_trans}
Let $(M, g)$ be a Riemannian manifold of dimension $n$, with Riemannian
metric $g$.
Let $f\in C^{\infty}(M)$ be a real-valued function and let $\tilde{g} = e^{2f}g$.
Let $Ric_{g}$ and $Ric_{\tilde{g}}$ denote Ricci curvature tensors with respect to $g$ and $\tilde{g}$ respectively. We then have the following formula:
\begin{equation*}
Ric_{\tilde{g}} = Ric_{g} - (n-2)(Hess_{g}(f) - df\otimes df) + (\Delta_gf
- (n-2)|df|_g^2)g.
\end{equation*}
\end{prop}

We will be using the above proposition in the following way:

\begin{prop}\label{Ric_bounded_conform_trans}
Let $(M, g)$ be an $n$-dimensional Riemannian manifold. Let $Q : M \rightarrow [1, \infty)$ be a smooth function satisfying
the following bounds:
\begin{itemize}
\item[(i)] $|dQ|_g \leq CQ^{1/4}$ for some constant $C \geq 0$,

\item[(ii)] $|Hess_{g}(Q)|_g \leq CQ^{-1/2}$ for some constant
$C \geq 0$.
\end{itemize}
Let $\tilde{g} = Q^{-3/2}g = e^{2Log(Q^{\frac{-3}{4}})}g$. Assume that $Ric_g\geq 0$. Then, $Ric_{\tilde{g}}\geq -K$, where $K\geq 0$ is some constant.

\end{prop}
\begin{proof}
For $f=Log(Q^{-3/4})$, we compute
\begin{equation*}
Hess_{g}(f) - df\otimes df = \frac{3}{4}\frac{1}{Q^2}dQ\otimes dQ -
\frac{3}{4}\frac{1}{Q}Hess_{g}(Q),
\end{equation*}
where we used proposition~\ref{Laplace_chain_rule}(ii). Applying the assumptions (i) and (ii), we obtain
\begin{equation*}
|Hess_{g}(f) - df\otimes df|_g \leq CQ^{-3/2}.
\end{equation*}
Appealing to proposition \ref{Laplace_chain_rule}(i), we have
\begin{equation*}
\Delta_gf = -\frac{3}{4}Q^{-2}|dQ|^2_g - \frac{3}{4}Q^{-1}\Delta_gQ.
\end{equation*}
Using the assumptions (i) and (ii) together with proposition \ref{Laplace_chain_rule}(iii),  we get
\begin{equation*}
|\Delta_gf| \leq CQ^{-3/2}.
\end{equation*}
Keeping in mind that $\langle\cdot,\cdot\rangle_{g}=Q^{3/2}\langle\cdot,\cdot\rangle_{\tilde{g}}$, the result then follows from proposition \ref{Ricci_curvature_conform_trans} and the assumption that
$Ric_g\geq 0$.
\end{proof}

\begin{prop}\label{Ric_bounded_conform_trans-2}
Let $(M, g)$ be an $n$-dimensional Riemannian manifold.
Let $Q : M \rightarrow [1, \infty)$ be a smooth function satisfying
the following bounds:
\begin{itemize}
\item[(i)] $|dQ|_g \leq CQ^{3/4}$ for some constant $C \geq 0$,

\item[(ii)] $|Hess_{g}(Q)|_g \leq CQ^{1/2}$ for some constant
$C \geq 0$.
\end{itemize}
Let $\tilde{g} = Q^{-1/2}g = e^{2Log(Q^{\frac{-1}{4}})}g$. Assume that $Ric_g\geq 0$. Then, $Ric_{\tilde{g}}\geq -K$, where $K\geq 0$ is some constant.
\end{prop}
\begin{proof} Using $f=Log(Q^{-1/4})$ and following the same pattern as in the proof of proposition~\ref{Ric_bounded_conform_trans}, we get the estimates
\begin{equation*}
|Hess_{g}(f) - df\otimes df|_g \leq CQ^{-1/2},
\end{equation*}
\begin{equation*}
|\Delta_gf| \leq CQ^{-1/2}.
\end{equation*}
Keeping in mind the rule $\langle\cdot,\cdot\rangle_{g}=Q^{1/2}\langle\cdot,\cdot\rangle_{\tilde{g}}$ and the assumption
$Ric_g\geq 0$, we get the result from proposition \ref{Ricci_curvature_conform_trans}.
\end{proof}

\end{appendices}

%------------------------------------------------------------------------

\end{document}